\documentclass{amsart}

\usepackage{amsmath,amssymb,amsfonts,amsthm}
\usepackage{wasysym}
\usepackage{xspace}
\usepackage{xcolor}
  \definecolor{myblue}{RGB}{17,83,166}
  \colorlet{myred}{red!50!black}
  \colorlet{mygreen}{green!50!black}
\usepackage{hyperref}
  \hypersetup{colorlinks=true,urlcolor=myred,linkcolor=myblue,citecolor=mygreen}
\usepackage[capitalise,noabbrev]{cleveref}
\usepackage[colorinlistoftodos,linecolor=purple,bordercolor=purple,backgroundcolor=purple!15,
  textsize=footnotesize]{todonotes}
\usepackage{enumitem}
  \setlist[enumerate]{label=(\roman*),
  align=left,leftmargin=1.75\parindent,labelwidth=1.75\parindent,labelsep=0pt}
  \setlist[itemize]{label={\raisebox{0.15ex}{\tiny \textbullet}},
  align=left,leftmargin=0.67\parindent,labelwidth=0.67\parindent,labelsep=0pt}

\theoremstyle{plain}
\newtheorem{theorem}{Theorem}[section]
\newtheorem{corollary}[theorem]{Corollary}
\newtheorem{lemma}[theorem]{Lemma}
\newtheorem{proposition}[theorem]{Proposition}
\newtheorem{definition}[theorem]{Definition}
\newtheorem{assumption}{Assumption}

\theoremstyle{definition}
\newtheorem{remark}[theorem]{Remark}
\newtheorem{example}[theorem]{Example}

\renewcommand{\geq}{\geqslant}
\renewcommand{\leq}{\leqslant}

\newcommand{\R}{\mathbb{R}}
\newcommand{\N}{\mathbb{N}}
\newcommand{\proba}{\mathbb{P}}
\newcommand{\TP}{\mathbb{TP}}
\newcommand{\Dcal}{\mathcal{D}}

\newcommand{\Scal}{\mathcal{S}}
\newcommand{\Tcal}{\mathcal{T}}
\newcommand{\Ucal}{\mathcal{U}}
\newcommand{\Vcal}{\mathcal{V}}
\newcommand{\Wcal}{\mathcal{W}}
\newcommand{\Xcal}{\mathcal{X}}
\newcommand{\Ycal}{\mathcal{Y}}
\newcommand{\Min}{{\scshape Min}\xspace}
\newcommand{\Max}{{\scshape Max}\xspace}
\newcommand{\state}{[n]}
\newcommand{\unit}{e}
\newcommand{\norm}[1]{\| \ifx\\#1\\ \cdot \else #1 \fi \|}
\newcommand{\supnorm}[1]{\| \ifx\\#1\\ \cdot \else #1 \fi \|_\infty}
\newcommand{\Hilbert}[1]{\| \ifx\\#1\\ \cdot \else #1 \fi \|_\textup{H}}
\newcommand{\Hnorm}{q_\textup{H}}
\newcommand{\toto}{\rightrightarrows}
\newcommand{\FP}{Fix}
\newcommand{\id}{Id}

\DeclareMathOperator{\dom}{dom}
\DeclareMathOperator{\range}{rge}
\DeclareMathOperator{\co}{co}
\DeclareMathOperator{\dist}{dist}
\DeclareMathOperator{\intr}{int}
\DeclareMathOperator{\clo}{cl}
\DeclareMathOperator{\supp}{supp}

\def\<#1,#2>{\langle #1, #2\rangle}

\title[Ergodic stochastic games]{An accretive operator approach to ergodic zero-sum stochastic games}
\author[A.\ Hochart]{Antoine Hochart}
\address{Universidad Adolfo Ib{\'a}{\~n}ez, Santiago, Chile}
\email{antoine.hochart@gmail.com}
\thanks{The author is funded by FONDECYT grant 3180662.
This work was initiated when the author was with Inria and CMAP, Ecole polytechnique,
supported by a PhD fellowship of Fondation Math\'ematique Jacques Hadamard.
The author was also partially supported by the Air Force Office of Scientific Research,
Air Force Material Command, USAF, under grant number FA9550-15-1-0500,
when he was with Toulouse School of Economics, Universit\'e Toulouse 1 Capitole.}

\date{\today}



\begin{document}

\maketitle

\begin{abstract}
  We study some ergodicity property of zero-sum stochastic games with a finite state space
  and possibly unbounded payoffs.
  We formulate this property in operator-theoretical terms, involving the solvability
  of an optimality equation for the Shapley operators (i.e., the dynamic programming operators)
  of a family of perturbed games.
  The solvability of this equation entails the existence of the uniform value, and its solutions
  yield uniform optimal stationary strategies.
  We first provide an analytical characterization of this ergodicity property, and address
  the generic uniqueness, up to an additive constant, of the solutions of the optimality equation.
  Our analysis relies on the theory of accretive mappings, which we apply to maps of the form
  $\id - T$ where $T$ is nonexpansive.
  Then, we use the results of a companion work to characterize the ergodicity of stochastic
  games by a geometrical condition imposed on the transition probabilities.
  This condition generalizes classical notion of ergodicity for finite Markov chains and
  Markov decision processes.
\end{abstract}

\section{Introduction}

\subsection{Motivations}

In zero-sum stochastic games, introduced by Shapley \cite{Sha53}, two agents are facing
repeatedly a zero-sum game that depends on a state variable, the evolution of which is
governed by a stochastic process jointly controlled by the players.
The finite-horizon value $v^k$ is an equilibrium payoff achieved when
the players are optimizing their average payoff over the $k$ first stages.
A significant part of the literature focuses on the study of asymptotic properties of $v^k$
as the horizon $k$ goes to infinity.
Two main approaches are particularly followed.
In the asymptotic approach, one studies the convergence of $v^k$ toward some limit
called the {\em asymptotic value}.
In the uniform approach, the problem is the existence of uniform optimal strategies, i.e.,
strategies that are near optimal in any $k$-stage game with $k$ large enough,
giving rise to the {\em uniform value}.
We refer the reader to \cite{NS03,MSZ15} for background on stochastic games.

When the state space is finite (which we assume throughout the paper), the existence of
the asymptotic value was proved for stochastic games with finite action spaces:
first for particular classes of games (recursive games by Everett \cite{Eve57},
absorbing games by Kohlberg \cite{Koh74}) and then in general by Bewley and Kohlberg \cite{BK76}.
The result still holds for recursive and absorbing games with compact action spaces and 
continuous payoff and transition functions (see \cite{Sor03} and \cite{RS01}, respectively).
On the other hand, Mertens and Neymann \cite{MN81} showed the existence of the uniform value
for stochastic games with finite action spaces.
Bolte, Gaubert and Vigeral \cite{BGV15} extended these results to ``definable'' stochastic
games, i.e., games whose data are definable in some o-minimal structure (e.g., when
the action spaces and the payoff and transition functions are semialgebraic).
The uniform value also exists for absorbing games with compact action spaces (see \cite{MNR09}).
We further mention that the existence of the uniform value is also guaranteed for Markov
decision processes (equivalent to one-player stochastic games) with bounded payoffs
(see Renault \cite{Ren11}).
However, the asymptotic value (hence the uniform value) does not exist for general
zero-sum stochastic games, even with standard assumptions on the data, i.e., compact action spaces,
continuous payoff and transition functions (see Vigeral \cite{Vig13}, see also
Ziliotto \cite{Zil16-AOP} for a counterexample in the framework of
zero-sum repeated games with signals).

Zero-sum stochastic games have a recursive structure which expresses itself in
their dynamic programming operators, a.k.a.\ Shapley operators.
From their analysis, it is possible to infer asymptotic properties of the games
(see e.g., Rosenberg and Sorin \cite{RS01}, Neymann \cite{Ney03}, Sorin \cite{Sor04},
Ziliotto \cite{Zil16-MOR}).
In this paper, following this so-called ``operator approach'', we focus on
the optimality equation (known as {\em average case optimality equation},
{\em Shapley equation} or {\em ergodicity equation}) $T(u) = \lambda \unit + u$,
where $T: \R^n \to \R^n$ is the Shapley operator of a game with $n$ states,
and $\unit$ denotes the unit vector of $\R^n$, i.e., the vector
whose coordinates are all equal to $1$.
Indeed, if the latter equation has a solution $(\lambda,u) \in \R \times \R^n$,
then the game has a uniform value, which is equal to $\lambda$ for every initial state.
Furthermore, the vector $u$ yields uniform optimal strategies.

Conditions which guarantee the solvability of the ergodic equation are usually
referred to as ergodicity conditions:
they are recurrence conditions which ensure some stability property of the state process.
For finite Markov chains (equivalent to zero-player games), the optimality equation
can be seen as an instance of the Poisson equation.
It is well known that it has a solution if the transition matrix has a unique invariant
probability measure, or equivalently if the Markov chain has a unique ergodic class.
For Markov decision processes (equivalent to one-player games), Bather \cite{Bat73}
showed that the ergodic equation has a solution when the system is communicating,
meaning that for each pair of states $(i,j)$, there is a stationary strategy such that
the probability to reach $j$ from $i$ in finite time is positive.
However, these accessibility relations do not easily carry over to the case of two players.
As a consequence, one usually imposes on the games a strong communication structure,
such as being irreducible (resp., unichain), that is, one requires the Markov chains
induced by {\em all pairs} of stationary strategies to be irreducible
(resp., unichain, i.e., to have a unique ergodic class) (see e.g., \cite{HK66,Vri03}).

In \cite{AGH15-DCDS}, Akian, Gaubert and Hochart obtained milder recurrence conditions
as a by-product of their analysis of the solvability of the ergodic equation.
However, the results only apply to stochastic games with {\em bounded} payoffs.
In applications, this boundedness property is restrictive and it is thus desirable to
extend the results to a broader framework.
This constitutes the purpose of the present work.

We further mention that the ergodic equation is a classical tool in the study of
stochastic games (as well as Markov decision processes) with a more general state space.
See e.g., \cite{BG93,Sen94,AHS97} for countable state spaces or
\cite{GB98,HLL00,JN01,Kue01} for Borel spaces.
Due to the technical difficulties inherent of these settings, the different kinds of
ergodicity assumptions made on these games are much more involved than those
previously discussed in the finite state space setting.
However, in all the papers mentioned above, these conditions all boil down to the irreducible
or the unichain case when the state space is finite.

\subsection{Description of the main results}

In this paper, we consider zero-sum stochastic games with a finite state space and
possibly {\em unbounded} payoffs that satisfy some ergodicity property.
Following Akian, Gaubert and Hochart \cite{AGH15-DCDS}, we formulate the latter condition
in operator-theoretical terms, namely, we say that a stochastic game is ergodic if
the optimality equation (a.k.a.\ ergodic equation) has a solution for all perturbations of
the payoff function that only depend on the state.
Let us note that, for finite Markov chains, this definition is one of several equivalent
manifestations of ergodicity, as discussed in \cite{AGH15-DCDS}.
One of our main result in this framework of stochastic games (\Cref{thm:ergodicity-slice-spaces})
is an analytical characterization of the latter stability property in terms of boundedness
in {\em Hilbert's seminorm} of some sets, so called {\em slice spaces}, which are invariant by
the Shapley operator $T$ of the game.
This stability result follows from a more general one (\Cref{thm:stability-acoe}) which is stated
in the framework of nonlinear Perron-Frobenius theory
(see \cite{Nus88,GG04}, or \cite{LN12} for background).

In a companion work (some of whose results were announced in \cite{AGH15-CDC}),
combinatorial criteria (of a graph-theoretical nature)
for the boundedness of all slice spaces of $T$ are given. 
Using this result, we then provide a geometrical ergodicity condition which is imposed
only to the transition probabilities of the game.
This condition involves two disjoint subsets of states, called {\em dominions}, each controlled
by a different player, in the sense that each player can make one of these subsets invariant
for the state process.
This communication structure readily generalizes classical ergodicity characterization of
finite Markov chains (see e.g., \cite{KS76}) or Markov decision processes
(see e.g., \cite{Bat73}).

Any Shapley operator $T$ is monotone (that is, order-preserving) and additively homogeneous
(that is, commutes with the addition by a vector proportional to the unit vector).
These properties imply in particular that $T$ is nonexpansive with respect to
the supremum norm, which makes $\id - T$ an $m$-accretive mapping.
Our results are derived from the study of the latter map.
Other results have been obtained using such an ``acccretive operator approach''
(see Vigeral \cite{Vig10} or Vigeral and Sorin \cite{SV16}), but contrary to them,
our analysis exploits the nonexpansiveness of $T$ with respect to Hilbert's seminorm.
Furthermore, in addition to the solvability of the ergodic equation, this approach
allows us to extend our analysis to the problem of uniqueness (up to an additive constant)
of the solutions.

The paper is organized as follows.
In \Cref{sec:preliminaries}, we recall the definition of zero-sum stochastic games,
we present the operator approach, and state our main results.
In \Cref{sec:accretive-mappings}, after some preliminaries on accretive mappings, we prove
a surjectivity condition for these mappings, which generalizes our ergodicity condition.
In \Cref{sec:FP-problems}, we infer from the previous section a result on the stability (under
additive perturbations) of existence of fixed points for nonexpansive maps.
We also address the ``generic'' uniqueness of the fixed point with respect to the space of
additive perturbations.
Finally, in \Cref{sec:applications}, we apply these results to stochastic games and
prove the two main results stated in \Cref{sec:preliminaries}.

We mention that part of these results were announced in the conference proceedings
\cite{Hoc16-MTNS}.

\section{Ergodicity of stochastic games: preliminaries and main results}
\label{sec:preliminaries}

\subsection{Zero-sum stochastic games}
\label{sec:stochastic-games}

In this paper, we consider zero-sum stochastic games with a {\em finite} state space and
possibly {\em unbounded} rewards.
They involve two players which we call \Max and \Min.
For the games to be well defined, we shall assume that their sets of actions are Borel spaces,
i.e., Borel subsets of a Polish space, and given any such Borel space $X$, we will denote by
$\Delta(X)$ the set of Borel probability measures on $X$, equipped with the weak* topology.

\medskip
A (zero-sum) stochastic game is a 7-tuple $\Gamma = (\state,A,B,K_A,K_B,r,p)$ defined by:
\begin{itemize}
  \item a finite state space $\state := \{1,\dots,n\}$;
  \item Borel action spaces $A$ and $B$ for
    players \Max and \Min, respectively;
  \item Borel admissible action sets $K_A \subset \state \times A$ and
    $K_B \subset \state \times B$:
    for each state $i \in \state$, $A_i = \{ a \in A \mid (i,a) \in K_A \}$ (resp.,
    $B_i = \{ b \in B \mid (i,b) \in K_B \}$) denotes the set of actions that are available
    to player \Max (resp., \Min) in state $i$.
    We let $K = \{ (i,a,b) \mid i \in \state , \; a \in A_i , \; b \in B_i \}$;
  \item a Borel measurable payoff function $r: K \to \R$;
  \item a Borel measurable transition function $p: K \to \Delta(\state)$.
\end{itemize}

A stochastic game $\Gamma$ is played in stages, starting from a given initial state
$i_0 \in \state$ known by the players.
It proceeds as follows: at each stage $\ell \geq 0$, if the current state is $i_\ell$,
the players choose simultaneously and independently some actions $a_\ell \in A_{i_\ell}$
and $b_\ell \in B_{i_\ell}$, respectively.
This incurs a payoff $r(i_\ell,a_\ell,b_\ell)$ given by player \Min to player \Max,
and the state $i_{\ell+1}$ at the next stage is drawn according to the probability
distribution $p(\cdot \mid i_\ell,a_\ell,b_\ell)$.
Then $(i_{\ell+1},a_\ell,b_\ell)$ is announced to both players.

\medskip
Denote by $H_k = K^{k-1} \times \state$ the set of histories of length $k \geq 1$.
A (behavioral) strategy of player \Max is a Borel measurable map
$\sigma: \cup_{k \geq 1} H_k \to \cup_{i \in \state} \Delta(A_i)$ such that, for every
$k \geq 1$ and every history $h_k = (i_1,a_1,b_1,\dots,i_k)$ of length $k$, we have
$\sigma(h_k) \in \Delta(A_{i_k})$.
In particular, the strategy $\sigma$ is stationary if $\sigma(h_k)$ only depends on
the final state $i_k$, let alone the history length $k$.
Likewise, a (behavioral) strategy of player \Min is defined by a Borel measurable map 
$\tau: \cup_{k \geq 1} H_k \to \cup_{i \in \state} \Delta(B_i)$.
We denote by $\Scal$ (resp., $\Tcal$) the set of strategies of player \Max (resp., \Min).
Also, for any state $i$, an element of $\Delta(A_i)$ (resp., $\Delta(B_i)$) is called
a mixed action of player \Max (resp., \Min), and when this action is a Dirac measure, 
we call it pure and identify it with its supporting point in $A_i$ (resp., $B_i$).

An initial state $i \in \state$ and a pair of strategies $(\sigma,\tau)$ of the players
induce a probability measure on the set of infinite histories, $H_\infty = K^\N$,
the expectation of which is denoted by $\mathbb{E}_{i,\sigma,\tau}$.
Then, for every length $k \geq 1$, the $k$-stage payoff is
\[
  \gamma_i^k(\sigma,\tau) = \mathbb{E}_{i,\sigma,\tau}
  \Bigg[ \frac{1}{k} \sum_{\ell=0}^{k-1} r(i_\ell,a_\ell,b_\ell) \Bigg]  ,
\]
and the $k$-stage game starting in $i$ has a value $v_i^k$ if
\begin{equation}
  \label{eq:finite-horizon-value}
  v_i^k = \sup_{\sigma \in \Scal} \, \inf_{\tau \in \Tcal} \; \gamma_i^k(\sigma,\tau)
  = \inf_{\tau \in \Tcal} \, \sup_{\sigma \in \Scal} \; \gamma_i^k(\sigma,\tau) .
\end{equation}

\medskip
Throughout the paper, we implicitly assume that all the finite-stage values of the games
that we consider exist.
We shall sometimes (but not always) make the following standing assumption, which ensures,
besides the existence of the finite-stage values, that the players have optimal strategies,
that is, the infima and suprema in \Cref{eq:finite-horizon-value} can be replaced by minima
and maxima, respectively (see \cite[Thm.~I.2.4]{MSZ15}).

\begin{assumption}
  \label{asm:compact-continuous}
  For every state $i \in \state$,
  \begin{enumerate}
    \item the action sets $A_i$ and $B_i$ are nonempty and compact;
    \item the payoff function $r(i,\cdot,\cdot)$ is bounded from below or from above;
      \label{asm:compact-continuous-ii}
    \item for every pair of actions $(a,b) \in A_i \times B_i$, the functions $r(i,\cdot,b)$
      and $p(j \mid i,\cdot,b)$, with $j \in [n]$, are upper semicontinuous (u.s.c.\ for short),
      and the functions $r(i,a,\cdot)$ and $p(j \mid i,a,\cdot)$, $j \in [n]$, are
      lower semicontinuous (l.s.c.\ for short).
  \end{enumerate}
\end{assumption}
%
%

\begin{remark}
  \label{rem:continuity}
  For all states $i \in [n]$ and all pairs of actions $(a,b) \in A_i \times B_i$, we have
  $\sum_{j \in [n]} p(j \mid i,a,b) = 1$.
  This implies that the upper (resp., lower) semicontinuity assumption of the functions
  $p(j \mid i,\cdot,b)$ (resp., $p(j \mid i,a,\cdot)$), with $j \in [n]$, is equivalent
  to their continuity.
\end{remark}

We emphasize that some of our results (in particular one of the main results,
\Cref{thm:ergodicity-slice-spaces}) do apply to more general settings, for which
the compactness of the action sets and/or the continuity of the payoff and transition
functions may not hold.
As an example, let us mention the following subclasses of stochastic games:
\begin{itemize}
  \item Markov decision processes (equivalent to one-player games);
  \item games with perfect information, for which each state $i$ is controlled by only one player,
    i.e., the functions $r(i,\cdot,\cdot)$ and $p(i,\cdot,\cdot)$ only depend on the actions of
    one player.
\end{itemize}
In any case, it worth noting that the payoff function {\em may not be bounded},
even if \Cref{asm:compact-continuous} holds
(see \Cref{ex:ergodicity-perfect-info,ex:ergodicity-dominions}).

\medskip
A standard problem is to determine if the value vector $v^k$ has a limit in $\R^n$ when
the horizon $k$ goes to infinity.
When this limit exists, it is called the {\em asymptotic value} of the game.
Bewley and Kohlberg \cite{BK76} proved that when the action spaces are finite
(along with the state space), the asymptotic value exists.
But this result cannot be extended to stochastic games with compact action spaces and
continuous payoff and transition functions (see \cite{Vig13}).

Loosely speaking, the asymptotic value $v = \lim_{k \to \infty} v^k$ exists when the players
are able to guarantee the same payoff, $v_i$, up to an arbitrarily small quantity,
in any $k$-stage game starting in $i$ with $k$ large enough and known in advance.
A stronger notion of value, the {\em uniform value}, asks for both players to guarantee, up to
any small perturbation, the same payoff in any finite-stage game with a sufficiently large
but unknown horizon.
Precisely, the stochastic game $\Gamma$ has a uniform value $v^\infty \in \R^n$ if,
for all $\varepsilon > 0$, there exist strategies $\sigma^* \in \Scal$,
$\tau^* \in \Tcal$, and a horizon $k_0 \in \N$ such that, for all initial states $i \in \state$,
for all strategies $\sigma \in \Scal$ and $\tau \in \Tcal$, and for all $k \geq k_0$,
\[
  \gamma_i^k(\sigma^*, \tau) \geq v^\infty_i - \varepsilon
  \quad \text{and} \quad
  \gamma_i^k(\sigma, \tau^*) \leq v^\infty_i + \varepsilon .
\]
The strategies $\sigma^*$ and $\tau^*$ are called uniform $\varepsilon$-optimal strategies,
or simply uniform optimal if they are uniform $\varepsilon$-optimal for every $\varepsilon > 0$.

Mertens and Neyman \cite{MN81} proved that the uniform value exists for stochastic games
with finite state and action spaces.
This result was extended in \cite{BGV15} to games with a finite state space for which
the data are definable in some o-minimal structure.
Note that when the uniform value exists, the asymptotic value also exists and they are equal.

\subsection{Operator approach}

Using a dynamic programming principle, Shapley proved in \cite{Sha53} that the value of
a finite-stage stochastic game satisfies a recursive formula.
This formula is written by means of the so-called {\em Shapley operator}.
Given a stochastic game $\Gamma$, the latter is a map $T: \R^n \to \R^n$ whose
$i$th coordinate is defined by
\begin{equation}
  \label{eq:Shapley-operator}
  \begin{split}
    T_i(x) & = \sup_{\mu \in \Delta(A_i)} \inf_{\nu \in \Delta(B_i)} \bigg\{
      r(i,\mu,\nu) + \sum_{\ell \in \state} x_\ell \, p(\ell \mid i,\mu,\nu) \bigg\} \\
    & = \inf_{\nu \in \Delta(B_i)} \sup_{\mu \in \Delta(A_i)} \bigg\{
      r(i,\mu,\nu) + \sum_{\ell \in \state} x_\ell \, p(\ell \mid i,\mu,\nu) \bigg\}, \quad
    x \in \R^n,
  \end{split}
\end{equation}
where, for any Borel measurable map $f: K \to \R$, we introduce its bilinear extension
\[
  f(i,\mu,\nu) := \int_{A_i} \int_{B_i} f(i,a,b) \, d\mu(a) \, d\nu(b), \quad
  \forall i \in \state, \quad
  \forall (\mu,\nu) \in \Delta(A_i) \times \Delta(B_i).
\]
The quantity $T_i(x)$ represents the value of a one-stage game with initial state $i \in \state$
and an additional payoff $x_j$ if the final state is $j \in \state$.
As mentioned in the latter subsection, we shall always assume -- implicitly -- that
the Shapley operator is well defined.
This is particularly the case when \Cref{asm:compact-continuous} holds and, if so,
the $\sup$ and $\inf$ operators in~\labelcref{eq:Shapley-operator} can be replaced by
$\max$ and $\min$, respectively. 

Then, the dynamic programming principle introduced by Shapley writes:
\begin{equation}
  \label{eq:Shapley-equation}
  v^0 = 0 \qquad \text{and} \qquad (k+1) v^{k+1} = T(k v^{k}), \quad \forall k \geq 0.
\end{equation}
As an immediate consequence, if the asymptotic value exists, then it is given by
\begin{equation}
  \label{eq:asymptotic-value}
  \lim_{k \to \infty} v^k = \lim_{k \to \infty} \frac{T^k(0)}{k} ,
\end{equation}
where $T^k := T \circ \dots \circ T$ denotes the $k$th iterate of $T$.

\medskip
The ``operator approach'' exploits the recursive structure~\labelcref{eq:Shapley-equation}
in order to infer convergence properties of $v^k$ from the analysis of the Shapley operator $T$
(see e.g., \cite{RS01,Ney03,Sor04,BGV15,Zil16-MOR}).
The latter analysis makes extensive use of the following properties, satisfied by
any Shapley operator:
\begin{flalign*}
  & \text{\em monotonicity:}
  && x \leq y \implies T(x) \leq T(y) , \quad x,y \in \R^n , && \\
  & \text{\em additive homogeneity:}
  && T(x + \alpha \unit) = T(x) + \alpha \unit , \quad x \in \R^n , \quad \alpha \in \R , && \\
  & \text{{\em nonexpansiveness}:}
  && \supnorm{T(x)-T(y)} \leq \supnorm{x-y} , \quad x, y \in \R^n , &&
\end{flalign*}
where $\R^n$ is endowed with its usual partial order, $\unit$ denotes the unit
vector of $\R^n$ and $\supnorm{}$ is the sup-norm of $\R^n$.
Note that the last property can be deduced from the first two (see \cite{CT80}).

A basic tool to study the asymptotic properties of the $k$-stage value is the {\em ergodic
equation}, also known as the {\em average cost optimality equation}
or sometimes {\em Shapley equation}:
\begin{equation}
  T(u) = \lambda \unit + u , \quad u \in \R^n, \quad \lambda \in \R ,
  \label{eq:acoe}
\end{equation}
where $\unit := (1,\dots,1)^\top$ is the unit vector of $\R^n$.
Intuitively, this equation means that if the players choose optimal strategies in
the game with an additional terminal payoff of $u_i$ when the last visited state is $i$,
then the payoff at each stage is equal to $\lambda$, whatever the horizon and
the initial state.
From this game-theoretical interpretation or, in the spirit of the operator approach,
using \labelcref{eq:asymptotic-value}, the additive homogeneity and the nonexpansiveness of $T$,
one easily deduce that if the ergodic equation has a solution, meaning that there exists a pair
$(\lambda,u) \in \R \times \R^n$ satisfying \Cref{eq:acoe}, then the asymptotic value of
the game exists and is a constant vector, whose entries are all equal to $\lambda$ -- we say
that a vector is ``constant'' if it is proportional to the unit vector.
Note that there is a unique scalar $\lambda$ for which \labelcref{eq:acoe} holds,
and for every $\alpha \in \R$, the pair $(\lambda, u + \alpha \unit)$ is also a solution
to~\labelcref{eq:acoe}: we say that $(\lambda,u)$ is defined up to an additive constant.

When the ergodic equation is solvable the existence of the asymptotic value may in fact
be refined by the following known result.
We state the proof for the reader's convenience.

\begin{proposition}[Existence of the uniform value]
  Let $\Gamma$ be a stochastic game such that the ergodic equation \labelcref{eq:acoe} is
  solvable for its Shapley operator $T$.
  Let $(\lambda,u)$ be any solution.
  Then, the uniform value exists and is equal to the constant vector $\lambda \unit$.
  Moreover, both players have stationary uniform $\varepsilon$-optimal strategies,
  which can be obtained from the vector $u$.
\end{proposition}

\begin{proof}
  Let $\varepsilon$ be a fixed positive real.
  For every state $i \in \state$, let $\mu^*_i \in \Delta(A_i)$ be an action of player \Max
  such that
  \[
    \inf_{\nu \in \Delta(B_i)} \bigg\{ r(i,\mu^*_i,\nu) +
      \sum_{\ell \in \state} u_\ell \, p(\ell \mid i,\mu^*_i,\nu) \bigg\} \geq
    T_i(u) - \varepsilon .
  \]
  Then, define the stationary strategy $\sigma^*$ of player \Max such that the mixed action
  $\mu^*_i$ is selected whenever the current state is $i$.

  We now introduce the Shapley operator $T^{\sigma^*}$ of a one-player game based on
  $\Gamma$ where the strategy of player \Max is fixed to $\sigma^*$.
  Precisely, the $i$th coordinate of $T^{\sigma^*}$ is given by
  \begin{equation*}
    T^{\sigma^*}_i(x) = \inf_{\nu \in \Delta(B_i)} \bigg\{ r(i,\mu^*_i,\nu) +
      \sum_{\ell \in \state} x_\ell \, p(\ell \mid i,\mu^*_i,\nu) \bigg\}, \quad x \in \R^n,
  \end{equation*}
  so that $T^{\sigma^*}(u) \geq T(u) - \varepsilon \unit = (\lambda - \varepsilon) \unit + u$.

  Denoting by $\supnorm{u} = \max_{i \in \state} |u_i|$ the sup-norm of $u$, we have, for all
  positive integers $k$ and all strategies $\tau$ of player \Min,
  \[
    \gamma^k (\sigma^*,\tau) \geq \frac{1}{k} \big( T^{\sigma^*} \big)^k (0) \geq
    \frac{1}{k}  \left( \big( T^{\sigma^*} \big)^k (u) - \supnorm{u} \, \unit \right),
  \]
  where the first inequality comes from the recursive property \labelcref{eq:Shapley-equation}
  applied to the one-player game with Shapley operator $T^{\sigma^*}$, and the second
  inequality stems from the monotonicity and the additive homogeneity of $T^{\sigma^*}$,
  along with the fact that $0 \geq u - \supnorm{u} \, \unit$.

  Let $k_0$ be such that $2 \supnorm{u} \leq k_0 \varepsilon$.
  By construction of $\sigma^*$, we deduce that for all integers $k \geq k_0$,
  for all state $i$, and for all strategy $\tau$ of player \Min, we have
  \[
    \gamma^k_i(\sigma^*,\tau) \geq \lambda - \varepsilon - \frac{2 \supnorm{u}}{k} \geq
    \lambda - 2 \varepsilon .
  \]
  With dual arguments, we show that there exists a stationary strategy $\tau^*$ of
  player \Min such that 
  \[
    \gamma^k_i(\sigma,\tau^*) \leq \lambda + \varepsilon + \frac{2 \supnorm{u}}{k} \leq
    \lambda + 2 \varepsilon
  \]
  for all $k \geq k_0$, for all $i \in \state$, and for all $\sigma \in \Scal$.
  This proves that $\lambda$ is the uniform value of $\Gamma$ for every initial state, and
  that $(\sigma^*,\tau^*)$ is a pair of stationary uniform $2 \varepsilon$-optimal strategies.
\end{proof}

\medskip
We conclude this subsection with a sufficient condition for the solvability of
the ergodic equation~\labelcref{eq:acoe}, hence for the existence of the uniform value.
To that purpose, let us introduce the following two definitions.
First, {\em Hilbert's seminorm} of a vector $x \in \R^n$ is defined by
\begin{equation}
  \label{eq:Hilbert-seminorm}
  \Hilbert{x} := \max_{i \in \state} x_i - \min_{i \in \state} x_i .
\end{equation}
Like any seminorm, it is a nonnegative function which is subadditive and absolutely
homogeneous (i.e., $\Hilbert{\alpha \, x} = |\alpha| \, \Hilbert{x}$ for
all $\alpha \in \R$ and all $x \in \R^n$).
However it fails to be a norm since $\Hilbert{x} = 0$ if and only if $x$ is proportional
to the unit vector of $\R^n$.
It is a standard result that any monotone and additively homogeneous self-map of $\R^n$
is nonexpansive with respect to this seminorm (see e.g., \cite{GG04}).
Second, given a map $T: \R^n \to \R^n$ and real numbers $\alpha, \beta$,
we define the {\em slice space}
\[
  \Scal_\alpha^\beta(T) := \{ x \in \R^n \mid \alpha \unit + x \leq T(x) \leq \beta \unit + x \} .
\]
Observe that if $T$ is monotone and additively homogeneous (in particular if $T$ is the Shapley
operator of a stochastic game), then any slice space is invariant by $T$.
Then we have the following.

\begin{theorem}[Solvability of the ergodic equation, {Cor.\ of \cite[Thm.~9]{GG04}}]
  \label{thm:solvability-acoe}
  Let $T: \R^n \to \R^n$ be a monotone and additively homogeneous map.
  If there exists a nonempty slice space bounded in Hilbert's seminorm, then the ergodic
  equation \labelcref{eq:acoe} is solvable.
\end{theorem}

\subsection{Ergodic stochastic games}
\label{sec:ergodic-stochastic-games}

Following \cite{AGH15-DCDS}, we introduce an operator-theoretical definition of ergodicity
for stochastic games.

\begin{definition}[Ergodicity of stochastic games]
  \label{def:ergodicity}
  A zero-sum stochastic game $\Gamma$ with Shapley operator $T$ is {\em ergodic} if
  for all perturbation vectors $g \in \R^n$, the ergodic equation \labelcref{eq:acoe} is
  solvable for $g+T$, i.e., for the Shapley operator of the perturbed game where,
  for every $(i,a,b) \in K$, the payoff is $g_i + r(i,a,b)$.
\end{definition}

\begin{example}
  \label{ex:basic-example}
  Let us illustrate our notion of ergodicity with the following very basic Shapley operators
  defined on $\R^2$ by
  \[
    T^\square(x) = \begin{pmatrix} x_1 \\ x_2 \end{pmatrix} , \quad 
    T^\ocircle(x) = \begin{pmatrix} x_2 \\ x_1 \end{pmatrix} , \quad 
    T^\bigtriangleup(x) = \begin{pmatrix} x_1 \vee x_2 \\ x_1 \wedge x_2 \end{pmatrix} ,
  \]
  where $\vee$ stands for $\max$ and $\wedge$ for $\min$.
  \begin{enumerate}[label=(\alph*),leftmargin=0pt,itemindent=1.5\parindent,labelwidth=1.5\parindent]
    \item For $g+T^\square$, the ergodic equation is solvable if and only if
      $g$ is a constant vector (and then any vector $u$ is a solution).
      Hence $T^\square$ is not ergodic (in the sense that the game whose Shapley operator
      is $T^\square$ is not ergodic).
    \item For $g+T^\ocircle$, the ergodic equation is solvable with every $g \in \R^2$.
      The solutions are then characterized by $\lambda = \frac{1}{2} (g_1 + g_2)$ and
      $u_1 - u_2 = \frac{1}{2} (g_1 - g_2)$.
      Hence $T^\ocircle$ is ergodic.
    \item For $g+T^\bigtriangleup$, the ergodic equation is solvable if and only if
      $g_1 \leq g_2$.
      If $g_1 < g_2$, then the solutions are characterized by
      $\lambda = \frac{1}{2} (g_1 + g_2)$ and $u_1 - u_2 = \frac{1}{2} (g_1 - g_2)$.
      If $g_1 = g_2$, the solutions satisfy $\lambda = g_1 = g_2$ and
      $u_1 \geq u_2$.
      Hence $T^\bigtriangleup$ is not ergodic.
  \end{enumerate}
\end{example}

In \cite{AGH15-DCDS}, several equivalent characterizations of ergodicity are given:
in terms of the recession operator
$\widehat{T}(x) := \lim_{\rho \to +\infty} \rho^{-1} T(\rho x)$
and the asymptotic value of the perturbed games (Theorem~3.1); in graph-theoretical terms
(Theorem~5.3); in game-theoretical terms (Proposition~5.1).
Note that, as discussed in the latter reference, all this equivalent criteria extend
the classical notion of ergodicity for finite Markov chains.
However, these results only apply to stochastic games with {\em bounded} payoffs.
They are all derived from a result that establishes a relation between ergodicity and the set
of fixed points of $\widehat{T}$ -- a game (with bounded payoffs) being ergodic if and only if
this fixed-point set is reduced to a line.
In general, this equivalence fails, as illustrated in \Cref{ex:ergodicity-perfect-info}, and so, the results
of \cite{AGH15-DCDS} cannot be readily extended, for instance to games with unbounded payoffs.
Thus, the wish to find a statement for ergodicity which applies to any kind of stochastic games
has motivated the formulation of \Cref{def:ergodicity}.
In this regard, one of our main results, which we state hereafter, can be seen as a generalization of
Theorem~3.1, {\em ibid.}

\begin{theorem}[Ergodicity and slice spaces]
  \label{thm:ergodicity-slice-spaces}
  A stochastic game is ergodic if and only if all the slice spaces of its Shapley operator
  are bounded in Hilbert's seminorm.
\end{theorem}

We mention that the novelty of this result lies in the necessity of the second statement.
Indeed, it is easy to deduce the sufficient part from \Cref{thm:solvability-acoe},
as the subsequent proof shows.

\begin{proof}[Proof -- ``if'' part]
  Let $\Gamma$ be a stochastic game for which all the slice spaces of its Shapley operator $T$
  are bounded in Hilbert's seminorm.
  Let $g \in \R^n$ be a state-dependent perturbation of the payoff function of $\Gamma$.
  We first notice that $0$ is in the slice space $\Scal_\alpha^\beta(g+T)$
  for $\beta = -\alpha = \supnorm{g+T(0)}$.
  Hence the latter is nonempty.
  Then, denoting by $m = \supnorm{g}$, we have for every $x \in \Scal_\alpha^\beta(g+T)$,
  \[
    (\alpha - m) \unit + x \leq -m \unit + g + T(x) \leq T(x) \leq
    m \unit + g + T(x) \leq (\beta + m) \unit + x .
  \]
  This proves that $\Scal_\alpha^\beta(g+T)$ is included in $\Scal_{\alpha-m}^{\beta+m}(T)$,
  hence bounded in Hilbert's seminorm. 
  It follows from \Cref{thm:solvability-acoe} that the ergodic equation \labelcref{eq:acoe}
  is solvable for $g+T$, and consequently that $\Gamma$ is ergodic.
\end{proof}

Let us now illustrate the result, first with the basic Shapley operators of
\Cref{ex:basic-example}, and then with a more involved stochastic game for which
the results in \cite{AGH15-DCDS} do not apply.

\begin{example}
  Consider the operators introduced in \Cref{ex:basic-example}.
  \begin{enumerate}[label=(\alph*),leftmargin=0pt,itemindent=1.5\parindent,labelwidth=1.5\parindent]
    \item The slice spaces of $T^\square$ are trivial, that is, $\Scal_\alpha^\beta(T^\square)$ is
      either empty if $\alpha > \beta$, or equal to the full space if $\alpha \leq \beta$.
      Hence, since $T^\square$ has slice spaces unbounded in Hilbert's seminorm: it is not ergodic.
    \item For $T^\ocircle$, the slice space $\Scal_\alpha^\beta(T^\ocircle)$ is nonempty if
      and only if $\alpha \leq 0 \leq \beta$.
      In that case, it is equal to the set of points $x$ such that
      $|x_1 - x_2| \leq -\alpha \wedge \beta$.
      Hence the slice space is bounded in Hilbert's seminorm by $-\alpha \wedge \beta$,
      which proves that $T^\ocircle$ is ergodic.
    \item For $T^\bigtriangleup$, if the slice space $\Scal_\alpha^\beta(T^\bigtriangleup)$
      is nonempty, i.e., if $\alpha \leq \beta$, then it contains all points $x$ such that
      $x_1 \geq x_2$.
      Hence it is not bounded in Hilbert's seminorm and $T^\bigtriangleup$ is not ergodic.
  \end{enumerate}
\end{example}

\begin{example}
  \label{ex:ergodicity-perfect-info}
  Consider the perfect-information stochastic game with two states, whose Shapley operator is
  given by
  \[
    T(x) = \left(
    \begin{gathered}
      \sup_{0 < p \leq 1} \big( 2 (1-p) + \log p + (1-p) x_1 + p x_2 \big) \\
      \inf_{0 < p \leq 1} \big( 2 (p-1) - \log p + p x_1 + (1-p) x_2 \big)
    \end{gathered}
    \right) , \quad x \in \R^2 .
  \]
  Player \Max controls the first state and player \Min the second.
  In state~1, player \Max chooses an action $p \in (0,1]$ which yields the current payoff
  $2 (1-p) + \log p$ and induces a probability $p$ to switch to state~2.
  The payoff in state~1 is positive for all $p \in (p_0,1)$ (where $p_0 \approx 0.2032$),
  it attains its maximum at $p = 1/2$ and tends to $-\infty$ as $p$ tends to $0$.
  Hence, securing a positive stage payoff entails a positive probability to switch to state~2,
  which is controlled by player \Min.
  A dual interpretation holds for player \Min in state~2.

  Letting $h: \R \to \R$ be the function defined by
  \[
    h(z) = \sup_{0 < p \leq 1} \big( 2 (1-p) + \log p + p z \big) = 
    \begin{cases}
      1 - \log(2-z)  \enspace & (z \leq 1) \\
      z \quad & (z \geq 1)
    \end{cases}, \quad z \in \R,
  \]
  we can write $T$ as 
  \[
    T(x) = \left(
    \begin{gathered}
      h(x_2-x_1) + x_1 \\
      - h(x_2-x_1) + x_2
    \end{gathered}
    \right) , \quad x \in \R^2 .
  \]
  It is then easy to check that all the slice spaces of $T$ are bounded in Hilbert's
  seminorm, hence that the game is ergodic.

  Furthermore, for all $x \in \R^2$ such that $x_2 \leq x_1$, the recession operator of
  $T$ is given by
  \[
    \widehat{T}(x) := \lim_{\rho \to +\infty} \frac{T(\rho x)}{\rho} = x .
  \]
  Thus, the fixed-point set of $\widehat{T}$ is not reduced to the line $\R \unit$,
  which shows that the results of \cite{AGH15-DCDS} do not apply to this game.
\end{example}

\medskip
In the conference paper \cite{AGH15-CDC} -- a preprint of a longer version is available
in \cite{AGH18} --
the authors give a combinatorial criterion for the boundedness in Hilbert's
seminorm of all the slice spaces of any monotone additively homogeneous self-map of $\R^n$.
The combination of that result, which involves a pair of directed hypergraphs, with
\Cref{thm:ergodicity-slice-spaces} provides a generalization of the graph-theoretical
ergodicity condition in \cite[Thm~5.3]{AGH15-DCDS}.
Similarly to Proposition~5.1, {\em ibid.}, we can then translate this condition in
game-theoretical terms.
With this end in view, let us introduce the notion of {\em dominion}, informally speaking
a subset of states that one player can make invariant for the state process.

\begin{definition}[Dominion]
  \label{def:dominions}
  Given a stochastic game $\Gamma$, we call {\em dominion} of player \Max (resp., \Min)
  a nonempty subset of states $D$ for which player \Max (resp., \Min) has a strategy
  such that from any initial position in $D$, the state remains almost surely in $D$
  at all stages, whatever strategy the other player chooses.
\end{definition}

We mention that an equivalent notion appeared in the algorithmic game theory literature,
first in \cite{GL89}, then in \cite{BEGM10}.
We further mention that the name ``dominion'' was introduced in \cite{JPZ08} 
with a definition slightly stronger that ours.

\begin{example}
  \label{ex:dominions}
  Let us illustrate the notion of dominion by considering the following stochastic game $\Gamma$.
  It has $3$ states and the action sets are $[0,1]$ for both players in all states.
  We next identify probability measures on $\R^3$ with vectors in the standard
  simplex of $\R^3$. 
  The transition probabilities of $\Gamma$ are given, for all actions $a \in [0,1]$ of
  player \Max and $b \in [0,1]$ of player \Min, by
  \[
    p(1,a,b) =
    \begin{pmatrix}
      \gamma a \\ b \, (1 - \gamma a) \\ (1-b) \, (1 - \gamma a)
    \end{pmatrix} \mkern-5mu , \;
    p(2,a,b) =
    \begin{pmatrix}
      \gamma b \\ a \, (1 - \gamma b) \\ (1-a) \, (1 - \gamma b)
    \end{pmatrix} \mkern-5mu , \;
    p(3,a,b) =
    \begin{pmatrix}
      0 \\ 0 \\ 1
    \end{pmatrix} \mkern-5mu ,
  \]
  where $\gamma$ is a fixed parameter in $(0,1)$.
  Since dominions are only defined by the dynamics of the state, it is not necessary to
  specify here the payoff function of $\Gamma$.
  It is then easy to check that the dominions of player \Max are
  \[
    \{3\} , \quad \{1,2,3\} ,
  \]
  whereas the dominions of player \Min are
  \[
    \{3\} , \quad \{1,3\} , \quad \{2,3\} , \quad \{1,2,3\} .
  \]
\end{example}

We can now give a game-theoretical characterization of ergodicity, which extends
\cite[Prop.~5.1]{AGH15-DCDS}.

\begin{theorem}[Ergodicity and dominions]
  \label{thm:ergodicity-dominions}
  A stochastic game satisfying \Cref{asm:compact-continuous} is ergodic if
  and only if the players do not have disjoint dominions.
\end{theorem}

\begin{remark}[The dominion condition for Markov chains and Markov decision processes]
  It is instructive to apply the latter theorem to zero-player or one-player games,
  that is, to finite Markov chains with rewards or Markov decision processes, respectively.
  \begin{enumerate}[label=(\alph*),leftmargin=0pt,itemindent=1.5\parindent,labelwidth=1.5\parindent]
    \item In the first case, the two players are dummies, meaning that they have only
      one possible action in each state.
      Any dominion (of any player) then contains at least one ergodic class of
      the Markov chain and, conversely, any ergodic class is a dominion for both players.
      Thus, the dominion condition in \Cref{thm:ergodicity-dominions} is equivalent to
      the classical ergodicity condition for finite Markov chains.
    \item For Markov decision processes, only one player is a dummy, say player \Min.
      Then, a dominion of this player is an absorbing subset of states for player \Max
      (meaning that the state remains almost surely in this set once it has reached it,
      whatever the strategy of player \Max is).
      In particular, it is also a dominion for player \Max.
      Since the intersection of two absorbing sets is itself absorbing, we deduce that
      the Markov decision process is ergodic if and only if there is a unique minimal
      nonempty absorbing set (with respect to the inclusion partial order) and
      this set has a nonempty intersection with any dominion of player \Max.
  \end{enumerate}
\end{remark}

\begin{example}
  Let us consider the games associated with the operators introduced in \Cref{ex:basic-example}.
  \begin{enumerate}[label=(\alph*),leftmargin=0pt,itemindent=1.5\parindent,labelwidth=1.5\parindent]
    \item For $T^\square$, this is the trivial Markov chain, where every state is
      absorbing.
      Hence, every state is a dominion of both players and thus the game is not ergodic.
    \item For $T^\ocircle$, this is also a Markov chain, which is irreducible.
      Hence, each player has only one dominion, which is the set of all state.
      Thus the game is ergodic.
    \item For $T^\bigtriangleup$, the nontrivial dominions are $\{1\}$ for player \Max and
      $\{2\}$ for player \Min.
      We deduce from \Cref{thm:ergodicity-dominions} that the game is not ergodic.
  \end{enumerate}
\end{example}

\begin{example}
  \label{ex:ergodicity-dominions}
  Consider the stochastic game $\Gamma$ (partially) introduced in \Cref{ex:dominions}.
  Assume that the payoff function is given by
  \[
    r(1,a,b) = \frac{a b}{a^3 + b^3} , \quad r(2,a,b) = \frac{- a b}{a^3 + b^3} , \quad
    r(3,a,b) = 0 ,
  \]
  for $(a,b) \neq (0,0)$ and $r(1,0,0) = r(2,0,0) = r(3,0,0) = 0$.
  Then, $\Gamma$ satisfies \Cref{asm:compact-continuous}, and since the players do not have
  disjoint dominions, it is ergodic.
  Further observe that, since the payoff function is unbounded, the results of \cite{AGH15-DCDS}
  cannot be applied.
\end{example}

\begin{remark}
  \label{rem:structural-property}
  \Cref{thm:ergodicity-dominions} suggests that ergodicity is a structural property, in the sense
  that it only depends on the transition function of the game, and not on the payoff function.
  We underline that this is true if and {\em only if} \cref{asm:compact-continuous} holds.
  Indeed, the transition structure of the game in \cref{ex:ergodicity-perfect-info} is essentially
  the same as the transition structure of a game with Shapley operator $T^\bigtriangleup$
  (\cref{ex:basic-example}): in the former example, if the payoffs were set to zero, the Shapley
  operator would be exactly equal to $T^\bigtriangleup$.
  However, the latter game is not ergodic, whereas the former (which does not satisfy
  \cref{asm:compact-continuous}) is.
\end{remark}

\section{Surjectivity of accretive mappings}
\label{sec:accretive-mappings}

The solvability of ergodic equation~\labelcref{eq:acoe} can be seen, up to an additive constant,
as a fixed-point problem.
In this perspective, it is useful to work in the quotient vector space $\R^n / \R \unit$.
Our definition of ergodicity (\Cref{def:ergodicity}) then boils down to the existence of
a fixed point for all additive perturbations in $\R^n / \R \unit$ of $[T]$, the quotiented
version of the Shapley operator $T$ or, equivalently, to the surjectivity of
the quotiented map $\id - [T]$, where $\id$ denotes the identity map.
It turns out that the quotiented map $[T]$ is nonexpansive, which implies that
$\id - [T]$ is accretive (see \Cref{def:accretive} below).
This motivates us to study the existence stability of a fixed point, under additive perturbations,
for nonexpansive maps in (finite-dimensional) Banach spaces, as well as the related
and more general problem of the surjectivity of accretive set-valued mappings.
The present section is dedicated to the study of the latter problem, whereas the former is
investigated in the next section.

In the remainder, $\Xcal$ refers to a finite-dimensional real vector space
equipped with a given norm, denoted by $\norm{}$.
Its dual space, $\Xcal^*$, is equipped with the dual norm, denoted by $\norm{}_*$, and
$\<\cdot,\cdot>$ refers to the duality product.

\subsection{Preliminaries on accretive mappings}

\subsubsection*{Set-valued analysis}
We recall here some basic definitions about set-valued mappings and refer the reader to
the monograph \cite{RW09} for more details on the subject.

Given a finite-dimensional real vector space $\Ycal$, a set-valued mapping
$A: \Xcal \toto \Ycal$ is a map sending each point of $\Xcal$ to a subset of $\Ycal$.
The {\em domain} of $A$ is the subset of $\Xcal$ defined by
$\dom(A) := \{ x \in \Xcal \mid A(x) \neq \emptyset \}$.
The {\em range} of $A$ is the subset of $\Ycal$ defined by
$\range(A) := \bigcup_{x \in \Xcal} A(x)$, and the {\em image} of any subset
$\Ucal \subset \Xcal$ is the subset of $\Ycal$ given by
$A(\Ucal) = \bigcup_{x \in \Ucal} A(x)$.

The {\em inverse} of $A$, denoted by $A^{-1}$, is the set-valued mapping from $\Ycal$ to $\Xcal$
sending any element $y \in \Ycal$ to the set $\{ x \in \Xcal \mid y \in A(x) \}$,
i.e., such that $x \in A^{-1}(y)$ if and only if $y \in A(x)$.
In particular, we have $(A^{-1})^{-1} = A$ and $\range(A) = \dom(A^{-1})$.
Also note that the image of any set $\Vcal \subset \Ycal$ by $A^{-1}$ is given by
$A^{-1}(\Vcal) = \{x \in \Xcal \mid A(x) \cap \Vcal \neq \emptyset \}$.

\medskip
We next define notions of continuity for set-valued mappings.
The {\em outer limit}, $\limsup_{x \to \bar{x}} A(x)$, and the {\em inner limit},
$\liminf_{x \to \bar{x}} A(x)$, of $A: \Xcal \toto \Ycal$ at any point $\bar{x} \in \Xcal$
are subsets of $\Ycal$ defined respectively by the following:
\begin{align*}
  y \in \limsup_{x \to \bar{x}} A(x) & \iff 
  \exists x_k \to \bar{x} , \enspace \exists y_k \to y , \enspace
  \forall k \in \N , \enspace y_k \in A(x_k) , \\
  y \in \liminf_{x \to \bar{x}} A(x) & \iff 
  \forall x_k \to \bar{x} , \enspace \exists y_k \to y , \enspace
  \exists k_0 \in \N , \enspace \forall k \geq k_0 ,
  \enspace y_k \in A(x_k) .
\end{align*}
Then we can define the notions of outer and inner semicontinuity for set-valued mappings.

\begin{definition}[Semicontinuity of set-valued mappings]
  \label{def:semicontinuity}
  A set-valued mapping $A: \Xcal \toto \Ycal$ is {\em outer semicontinuous} (o.s.c.\ for short)
  at $\bar{x} \in \Xcal$ if
  \[
    \limsup_{x \to \bar{x}} A(x) \subset A(\bar{x})
  \]
   and {\em inner semicontinuous} (i.s.c.\ for short) at $\bar{x} \in \Xcal$ if
  \[
    \liminf_{x \to \bar{x}} A(x) \supset A(\bar{x}) .
  \]
  It is {\em continuous} at $\bar{x}$ if it is both outer and inner semicontinuous.

  These notions are invoked relative to any subset $\Ucal$ of $\Xcal$ containing $\bar{x}$
  if the properties hold when restricting the convergence $x \to \bar{x}$ to $\Ucal$, i.e.,
  when all the sequences $x_k \to \bar{x}$ are required to lie in $\Ucal$.
\end{definition}

Note that if $A$ is i.s.c.\ at a point $x \in \dom(A)$, then
$x \in \intr \left( \dom(A) \right)$, the interior of $\dom(A)$.
Let us further mention that outer semicontinuity differs from upper semicontinuity,
another notion commonly found in the literature (see e.g., \cite{AF09}).
However, when a set-valued mapping is locally bounded, these two definitions agree.
As for lower semicontinuity, it is equivalent to inner semicontinuity.
The reader can find a discussion on these aspects in \cite[Ch.~5]{RW09} (see in particular
the Commentary Section).

\subsubsection*{Duality mapping}
\label{sec:duality-mapping}
Duality mappings are set-valued mappings that appears in the study of Banach spaces (see e.g.,
\cite{Pet70}) or in applications involving nonexpansive and monotone-like operators (e.g.,
evolution equations \cite{Bro76}, fixed-point approximation \cite{Rei94}). 
In this paper, we only consider {\em normalized duality mappings} and refer the reader to \cite{Cio90}
and the references therein for a general view on the subject.

\begin{definition}[Duality mapping]
  The {\em (normalized) duality mapping} on the normed space $(\Xcal, \norm{})$ is
  the set-valued mapping $J: \Xcal \toto \Xcal^*$ defined by
  \[
    J(x) = \{x^* \in \Xcal^* \mid \norm{x^*}_* = \norm{x} , \;
    \<x,x^*> = \norm{x}^2 \}, \quad x \in \Xcal .
  \]
\end{definition}

Note that, by the Hahn-Banach separation theorem, $\dom(J) = \Xcal$, i.e., $J(x)$ is nonempty
for every vector $x \in \Xcal$.
Furthermore, Asplund characterized in \cite[Thm.~1]{Asp67} the image of any point by a duality
mapping as the subdifferential at this point of some convex function.
This  entails that $J(x)$ is a compact convex subset of $\Xcal^*$ for every $x \in \Xcal$.
Also, it readily stems from the definition that $J$ is homogeneous of degree one, i.e.,
for every $x \in \Xcal$ and every $\lambda \in \R$, we have $J(\lambda x) = \lambda J(x)$.
Finally, a straightforward application of the definitions leads to the following lemma.

\begin{lemma}[Semicontinuity of duality mappings]
  \label{lem:osc-duality-mapping}
  Any normalized duality mapping on a finite-dimensional vector space is outer semicontinuous.
  \qed
\end{lemma}

\begin{example}
  Let $\Xcal = \R^n$.
  If $\Xcal$ is equipped with the standard Euclidean norm, then $J$ is the identity map.
  More generally, if $\Xcal$ is equipped with an $L^p$ norm with $1 < p < +\infty$,
  then $J$ is single-valued and given for all $x \neq 0$ by
  \[
    J(x) = \frac{\norm{x}_p}{\norm{x}_q} \, x
  \]
  where $q$ is the positive real number defined by $p^{-1} + q^{-1} = 1$.
\end{example}

\begin{example}
  Assume that the norm $\norm{}$ on $\Xcal$ is polyhedral (e.g., an $L^1$ or
  $L^\infty$ norm), meaning that there is a finite symmetric family
  $\Wcal \subset \Xcal^*$ of linear forms on $\Xcal$ such that
  \[
    \norm{x} = \max_{x^* \in \Wcal} \<x,x^*> , \quad \forall x \in \Xcal .
  \]
  For instance, when $\Xcal$ is the Euclidean space $\R^n$ and $\{ \unit_i \}_{i \in [n]}$
  denotes its standard basis, the family
  $\Wcal_\infty = \{ \varepsilon_i \unit_i \mid i \in [n] , \; \varepsilon_i = \pm 1\}$
  defines the standard $L_\infty$ norm, whereas the family
  $\Wcal_1 = \left\{ \sum_{i = 1}^n \varepsilon_i \unit_i \mid
  (\varepsilon_i)_{i \in [n]} \in \{ \pm 1 \}^n \right\}$ defines the $L_1$ norm.
  Then, one may check that, for all $x \in \Xcal$,
  \[
    J(x) = \norm{x} \, \co \big\{ x^* \in \Wcal \mid \<x,x^*> = \|x\| \big\} ,
  \]
  where $\co(\Ucal)$ denotes the convex hull of any subset $\Ucal$ of a vector space.
\end{example}

\subsubsection*{Accretivity}
\label{sec:accretivity}
Accretive operators are generalization in Banach spaces of monotone operators in Hilbert spaces.
They appear in particular in the study of nonlinear evolution equations (see e.g., \cite{Bro76}),
and more recently in game theory (\cite{Vig10,SV16}). 

\begin{definition}[Accretive mappings]
  \label{def:accretive}
  A set-valued mapping $A: \Xcal \toto \Xcal$ is {\em accretive} if
  \[
    \forall x, y \in \Xcal, \quad \forall u \in A(x), \quad \forall v \in A(y), \quad
    \exists x^* \in J(x-y), \quad \<u-v,x^*> \geq 0 .
  \]
  If, in addition, $\range (\id + A) = \Xcal$, where $\id$
  denotes the identity map on $\Xcal$, then $A$ is {\em $m$-accretive}.
\end{definition}

\begin{definition}[Coaccretive mappings]
  \label{def:coaccretive}
  A set-valued maping $A: \Xcal \toto \Xcal$ is {\em coaccretive} if $A^{-1}$ is
  accretive, that is, if
  \[
    \forall x, y \in \Xcal, \quad \forall u \in A(x), \quad \forall v \in A(y), \quad
    \exists x^* \in J(u-v), \quad \<x-y,x^*> \geq 0 .
  \]
\end{definition}

Let us mention that in a Hilbert space, the normalized duality mapping is the identity map,
so that accretive mappings coincide with monotone operators, whereas
$m$-accretive mappings coincide with maximally monotone operators.
In particular, for functions form $\R$ to $\R$, accretive means nondecreasing, whereas
$m$-accretive means continuous and nondecreasing.
To get an intuition about the results in this Section, the reader can think of
this special case.

We conclude this subsection with few remarks about the latter definitions.

\begin{remark}
  \label{rem:accretive-coaccretive}
  \begin{enumerate}[label=(\alph*),leftmargin=0pt,itemindent=1.5\parindent,labelwidth=1.5\parindent]
    \item In the literature, one can also find an equivalent definition of accretivity which
      does not make use of the duality mapping:
      a set-valued mapping $A: \Xcal \toto \Xcal$ is accretive if and only if for every
      $x, y \in \Xcal$, every $u \in A(x)$ and $v \in A(y)$, and every $\lambda > 0$,
      we have $\norm{x-y} \leq \norm{x - y + \lambda (u - v)}$
      (see \cite[Ch.~VI, Prop.~1.3]{Cio90}).
    \item It is known that if $A: \Xcal \toto \Xcal$ is accretive, then
      $\range(\id + \lambda A) = \Xcal$ for {\em every} $\lambda > 0$ if and only if
      $\range(\id + \lambda A) = \Xcal$ for {\em some} $\lambda > 0$
      (see Ch.~VI, Prop.~1.6, {\em ibid.}).
  \end{enumerate}
\end{remark}

\subsection{Surjectivity conditions}

\subsubsection*{Local boundedness of coaccretive mappings}
A set-valued mapping $A: \Xcal \toto \Xcal$ is locally bounded at a point $x \in \Xcal$
if there exists a neighborhood $\Ucal \subset \Xcal$ of $x$ such that $A(\Ucal)$ is bounded.
It is known that in a reflexive Banach space, an accretive mapping is locally bounded
at any point in the interior of its domain (see \cite{FHK72}).
The following result (which is new, as far as we know) shows that this property also
holds for coaccretive mappings, at least in finite dimension.

\begin{proposition}[Local boundedness]
  \label{prop:local-boundedness}
  Let $(\Xcal,\norm{})$ be a finite-dimensional real vector space and let
  $A: \Xcal \toto \Xcal$ be a coaccretive mapping.
  Then, for every point $x$ in the interior of $\dom(A)$, $A$ is locally bounded at $x$.
\end{proposition}

\begin{proof}
  Toward a contradiction, assume that $A$ is not locally bounded at a point $x$
  in the interior of $\dom(A)$, that is, for all neighborhoods $\Ucal$ of $x$,
  $A(\Ucal)$ is not bounded.
  Then, there exists a sequence $(x_k, y_k)_{k \in \N}$ in $\Xcal \times \Xcal$ such that
  $(x_k)_{k \in \N}$ converges to $x$, $\norm{y_k}$ tends to infinity, and 
  $y_k \in A(x_k)$ for all integers $k$.
  We may assume that $\norm{y_k} > 0$ for all $k \in \N$.
  Since the dimension is finite, the bounded sequence $(y_k / \norm{y_k})_{k \in \N}$
  has a convergent subsequence.
  Let $(y_{n_k} / \norm{y_{n_k}})_{k \in \N}$ be such a subsequence, which converges
  toward some point $w \in \Xcal$.
  
  The point $x$ being in the interior of $\dom(A)$, there exists a scalar $\alpha > 0$
  such that $x + \alpha w \in \dom(A)$.
  Let $\bar{y} \in A(x + \alpha w)$.
  Since $A$ is coaccretive, then for all integers $k$ there is a linear form
  $y^*_k \in J(y_k - \bar{y})$, where $J$ is the duality mapping on $(\Xcal,\norm{})$,
  such that $ \< x_k - ( x + \alpha w), y^*_k > \geq 0$.
  By homogeneity of the duality mapping $J$, for all integers $k$ we have
  \[
    w^*_k = \frac{y^*_k}{\norm{y_k}} \in J \left( \frac{y_k - \bar{y}}{\norm{y_k}} \right) .
  \]
  In particular, the sequence $(w^*_k)_{k \in \N}$ is bounded.
  We also have, for all $k \in \N$,
  \begin{equation}
    \label{eq:accretive-ineq}
    \< x_k - x - \alpha w, w^*_k > \geq 0 .
  \end{equation}
  Let $w^* \in \Xcal^*$ be a cluster point of the bounded subsequence $(w^*_{n_k})_{k \in \N}$.
  Since the subsequence $((y_{n_k}-\bar{y}) / \norm{y_{n_k}})_{k \in \N}$ converges
  toward $w$, and since $J$ is o.s.c.\ (\Cref{lem:osc-duality-mapping}), then we deduce that
  $w^* \in J(w)$.
  On the other hand, \Cref{eq:accretive-ineq} yields $\<w,w^*> \leq 0$,
  a contradiction since $\<w,w^*> = \|w\|^2 = 1$ by definition of the duality mapping.
\end{proof}

By application of the Heine-Borel property for compact sets, we readily get the following.

\begin{corollary}
  \label{cor:compact-image}
  Let $(\Xcal,\norm{})$ be a finite-dimensional real vector space and let
  $A: \Xcal \toto \Xcal$ be a coaccretive mapping.
  Then, the image by $A$ of any compact subset of $\Xcal$ included in the interior of
  $\dom(A)$ is bounded.
\end{corollary}

\subsubsection*{Characterization of surjectivity}
Let us denote by $\dist(0,\Ucal)$ the distance of the origin to any set
$\Ucal \subset \Xcal$, i.e., $\dist(0, \Ucal) = \inf_{x \in \Ucal} \norm{x}$.
We next give a sufficient condition for an $m$-accretive mapping to be surjective.

\begin{theorem}[Sufficient condition of surjectivity, see {\cite[Cor.\ of Thm.~3]{KS80}}]
  \label{thm:m-accretive-surjectivity}
  Let $(\Xcal,\norm{})$ be a finite-dimensional real vector space and let
  $A: \Xcal \toto \Xcal$ be an $m$-accretive set-valued mapping.
  Assume that
  \[
    \lim_{\norm{x} \to \infty} \dist(0, A(x)) = +\infty .
  \]
  Then, $\range(A) = \Xcal$.
\end{theorem}

We now state the main result of this section.

\begin{theorem}[Surjectivity of accretive mappings]
  \label{thm:surjectivity-condition}
  Let $(\Xcal,\norm{})$ be a finite-dimensional real vector space and let
  $A: \Xcal \toto \Xcal$ be an accretive mapping.
  If $\range(A) = \Xcal$ then, for all scalars $\alpha \geq 0$, the set
  \[
    \Dcal_\alpha := \{x \in \Xcal \mid \dist(0, A(x)) \leq \alpha \}
  \]
  is bounded.
  Moreover, if $A$ is $m$-accretive, then the two properties are equivalent.
\end{theorem}

\begin{proof}
  Assume that $\range(A) = \Xcal$.
  Equivalently, we have $\dom(A^{-1}) = \Xcal$.
  Moreover, since $A$ is accretive, its inverse $A^{-1}$ is coaccretive
  (see \Cref{sec:accretivity}).
  Hence, according to \Cref{cor:compact-image}, the image of any compact set
  by $A^{-1}$ is bounded.

  Let $\alpha, \alpha'$ be two nonnegative real numbers such that $\alpha < \alpha'$.
  If $x \in \Dcal_\alpha$, then $A(x) \cap B(0; \alpha') \neq \emptyset$, where $B(0; \alpha')$
  denotes the closed ball centered at $0$ of radius $\alpha'$.
  This proves that $\Dcal_\alpha \subset A^{-1}(B(0; \alpha')) =
  \{ x \in \Xcal \mid A(x) \cap B(0; \alpha') \neq \emptyset \}$.
  Hence $\Dcal_\alpha$ is bounded.

  For the converse, it is readily seen that the coercivity condition in
  \Cref{thm:m-accretive-surjectivity} is equivalent to the boundedness of
  all the sets $\Dcal_\alpha$.
  Hence the result, when $A$ is $m$-accretive..
\end{proof}

\section{Fixed point problems for nonexpansive maps}
\label{sec:FP-problems}

In this section, we use the main result of the previous one to study problems related to
the existence stability of fixed points of nonexpansive maps.
We use the same notation as before.
In particular, $(\Xcal, \norm{})$ shall refer to a finite-dimensional real normed space.
We may sometimes identify a map $A: \Xcal \to \Xcal$
with the set-valued mapping sending every $x \in \Xcal$ to $\{ A(x) \}$.
Furthermore, recall that a map $T: \Xcal \to \Xcal$ is nonexpansive (with respect to $\norm{}$)
if $\norm{T(x)-T(y)} \leq \norm{x-y}$ for all $x, y \in \Xcal$.

\subsection{Existence stability under additive perturbations}

Let us first recall the classical link between nonexpansive maps and accretive mappings.
We give the proof for the reader's convenience.

\begin{lemma}
  \label{lem:nonexpansive-accretive}
  If $T: \Xcal \to \Xcal$ is a nonexpansive map, then the mapping
  $A = \id - T$ is $m$-accretive.
\end{lemma}

\begin{proof}
  We first show that $A$ is accretive.
  Let $x, y \in \Xcal$ and $x^* \in J(x-y)$, where $J$ is the duality mapping
  on $(\Xcal, \norm{})$.
  We have
  \[
    \<T(x)-T(y), x^*> \leq \norm{x^*}_* \, \norm{T(x)-T(y)} \leq \norm{x^*}_* \, \norm{x-y} =
    \<x-y,x^*>
  \]
  where the first inequality stems from the definition of the dual norm, the second inequality
  from the nonexpansiveness of $T$, and the equality comes from the definition of $J$.
  We deduce that $\<A(x)-A(y), x^*> = \<x-y-(T(x)-T(y)), x^*> \geq 0$, which proves that
  $A$ is accretive.

  Now, let $\lambda > 0$ and $z \in \Xcal$.
  One can easily check that the point $z$ is in the range of $\id + \lambda A$ if
  and only if the map $x \mapsto T_{\lambda, z}(x) =
  \frac{\lambda}{1+\lambda} T(x) + \frac{1}{1+\lambda} z$ has a fixed point.
  Since $T$ is nonexpansive, $T_{\lambda, z}$ is a contraction.
  More precisely, for all $x, y \in \Xcal$, we have
  \[
    \left\| T_{\lambda,z}(x) - T_{\lambda,z}(y) \right\| \leq
    \frac{\lambda}{1+\lambda} \|x-y\| \enspace ,
  \]
  with $\frac{\lambda}{1+\lambda} < 1$.
  Hence, by the Banach fixed-point theorem, $T_{\lambda,z}$ has a (unique) fixed point.
  This proves that $\range (\id +\lambda A) = \Xcal$ for every $\lambda > 0$,
  and so that $A$ is $m$-accretive.
\end{proof}

The following corollary of \Cref{thm:surjectivity-condition} provides a necessary and
sufficient condition for the existence of a fixed point for all additive perturbations
of a nonexpansive map.

\begin{corollary}[Existence stability of a fixed point]
  \label{cor:stability-FP}
  Let $(\Xcal,\norm{})$ be a finite-dimensional real vector space and let
  $T: \Xcal \to \Xcal$ be a nonexpansive map.
  Then, the following are equivalent:
  \begin{enumerate}
    \item for every vector $g \in \Xcal$, the map $g+T$ has a fixed point;
      \label{it:stability-FP-i}
    \item every nonexpansive map $G: \Xcal \to \Xcal$ such that
      $\sup_{x \in \Xcal} \norm{G(x)-T(x)} < \infty$ has a fixed point;
      \label{it:stability-FP-ii}
    \item for every scalar $\alpha \geq 0$, the set
      $\Dcal_\alpha(T) = \{x \in \Xcal \mid \norm{x-T(x)} \leq \alpha \}$ is bounded.
      \label{it:stability-FP-iii}
  \end{enumerate}
\end{corollary}

\begin{proof}
  First, observe that a vector $g$ is in the range of $\id - T$ if and only if
  the map $g+T$ has a fixed point.
  Thus, the equivalence between \Cref{it:stability-FP-i} and \Cref{it:stability-FP-iii}
  is a mere application of \Cref{thm:surjectivity-condition} to $\id-T$,
  which is $m$-accretive according to \Cref{lem:nonexpansive-accretive}.
  Second, it is straightforward to check that
  \labelcref{it:stability-FP-ii}~$\Rightarrow$~\labelcref{it:stability-FP-i}.

  We now prove that \labelcref{it:stability-FP-iii}~$\Rightarrow$~\labelcref{it:stability-FP-ii}.
  Assume that \Cref{it:stability-FP-iii} holds and let $G: \Xcal \to \Xcal$ be
  a nonexpansive map such that $\sup_{x \in \Xcal} \norm{G(x)-T(x)} \leq M$ for some $M > 0$.
  One can readily check that $\Dcal_\alpha(G) \subset \Dcal_{\alpha+M}(T)$ for every
  $\alpha \geq 0$.
  Hence all the sets $\Dcal_\alpha(G)$ are bounded.
  Since we have already proved that
  \ref{it:stability-FP-i}~$\Leftrightarrow$~\ref{it:stability-FP-iii}, by applying
  the equivalence to $G$ we deduce in particular that $G$ has a fixed point.
\end{proof}

\subsection{Uniqueness condition}

Given a nonexpansive map $T: \Xcal \to \Xcal$, let us introduce the set-valued mapping
$\FP: \Xcal \toto \Xcal$ defined by
\begin{equation}
  \label{eq:FP-mapping}
  \FP(g) = \{ x \in \Xcal \mid g + T(x) = x \} , \quad g \in \Xcal ,
\end{equation}
that is, the mapping that sends each vector $g \in \Xcal$ to the set of fixed points of $g+T$.
Observe that the inverse mapping of $\FP$ is
\[
  \FP^{-1} = \id-T ,
\]
so that $\FP$ is coaccretive by \Cref{lem:nonexpansive-accretive}. 
By a straightforward application of the definition, we have the following.

\begin{lemma}
  \label{lem:FP-osc}
  The fixed-point mapping $\FP$ defined in \labelcref{eq:FP-mapping} is outer semicontinuous.
  In particular it is closed-valued.
  \qed
\end{lemma}

We shall need the following technical lemma, which is a variant of the Hahn-Banach
separation theorem.

\begin{lemma}
  \label{lem:separation}
  Let $J$ be the duality mapping on the finite-dimensional normed space $(\Xcal,\norm{})$,
  and let $x$ be any vector in $\Xcal$.
  Then,
  \[
    x \neq 0 \iff \exists w \in \Xcal \setminus \{0\} , \quad
    \forall x^* \in J(x) , \quad \<w,x^*> > 0 .
  \]
\end{lemma}

\begin{proof}
  Let $x \in \Xcal \setminus \{0\}$.
  We know that $J(x)$ is a compact convex subset of $\Xcal^*$ (see \Cref{sec:duality-mapping}).
  Furthermore, $0 \notin J(x)$ by definition.
  Hence, according to the Hahn-Banach separation theorem, there exists an affine hyperplane
  of $\Xcal^*$ strongly separating the two compact convex sets $J(x)$ and $\{0\}$,
  i.e., there exists some vector $w \in \Xcal \setminus \{0\}$ and a constant $\varepsilon > 0$
  such that, for all $x^* \in J(x)$, we have $\<w,x^*> \geq \varepsilon \geq 0$.

  Conversely, if $x = 0$, then $J(x) = \{0\}$ and so, for all $w \in \Xcal$, we have
  $\<w,x^*> = 0$ with $x^* = 0 \in J(x)$.
\end{proof}

We now state the main result of this subsection.

\begin{theorem}[Uniqueness of the fixed point]
  \label{thm:FP-uniqueness}
  Let $(\Xcal,\norm{})$ be a finite-dimensional real vector space and let
  $T: \Xcal \to \Xcal$ be a nonexpansive map.
  Then, the fixed-point mapping $\FP: \Xcal \toto \Xcal$ defined in
  \labelcref{eq:FP-mapping} is continuous at $g \in \dom(\FP)$ if
  and only if $g \in \intr \left( \dom(\FP) \right)$ and $\FP(g)$ is a singleton, i.e.,
  $g+T$ has a unique fixed point.
\end{theorem}

\begin{proof}
  Suppose first that the mapping $\FP$ is single-valued at
  $g \in \intr \left( \dom(\FP) \right)$ and denote by $\bar{x}$ the unique fixed point of $g+T$.
  Since $\FP = (\id-T)^{-1}$ is coaccretive, we know by \Cref{prop:local-boundedness}
  that $\FP$ is locally bounded at $g$.
  Hence there is a neighborhood $\Ucal$ of $g$ such that $\FP(\Ucal)$ is bounded.
  We may further assume that $\Ucal$ is included in $\dom(\FP)$ since
  $g \in \intr \left( \dom(\FP) \right)$.

  We next show that $\FP$ is i.s.c.\ at $g$.
  Let $(g_k)_{k \in \N}$ be any sequence in $\Ucal$ that converges to $g$.
  Since $\Ucal \subset \dom(\FP)$, for every $k \in \N$ there exists some $x_k \in \FP(g_k)$.
  The sequence $(x_k)_{k \in \N}$, being in $\FP(\Ucal)$, is bounded.
  Let $x$ be any cluster point.
  By continuity of $T$, we necessarily have $x \in \FP(g)$, hence $x = \bar{x}$.
  This proves that the sequence $(x_k)_{k \in \N}$ converges to $\bar{x}$, and so, that
  $\FP$ is i.s.c.\ at $g$.
  Since $\FP$ is also o.s.c.\ at $g$ (\Cref{lem:FP-osc}), we deduce that
  it is continuous at $g$.

  Conversely, suppose that $\FP$ is continuous at a point $g \in \dom(\FP)$.
  Then, by definition of inner semicontinuity, $g \in \intr \left( \dom(\FP) \right)$.
  Let $x$ and $y$ be two points in $\FP(g)$.
  Let $w \in \Xcal \setminus \{0\}$ and for every positive integer $k$,
  define $g_k = g - k^{-1} w$.
  We may assume, without loss of generality, that $g_k$ is in $\dom(\FP)$ for all $k$.
  Since $\FP$ is continuous at $g$, hence inner semicontinuous, there exists a sequence
  of elements $x_k \in \FP(g_k)$ converging to $x$.
  Furthermore, since $\FP$ is coaccretive, for all $k \in \N$ there exists
  a point $x^*_k \in J(x_k - y)$ (where $J$ is the duality mapping on $(\Xcal,\norm{})$)
  such that $\< g_k - g, x^*_k > \geq 0$, which yields $\< w, x^*_k > \leq 0$.

  Let $x^* \in \Xcal^*$ be a cluster point of the bounded sequence $(x^*_k)_{k \in \N}$.
  From the latter inequality, we get $\<w, x^*> \leq 0$.
  Moreover, since the duality mapping $J$ is o.s.c.\ (\Cref{lem:osc-duality-mapping}),
  we also have $x^* \in J(x-y)$.
  Thus, we have proved that for any point $w \in \Xcal \setminus \{0\}$,
  there exists an element $x^* \in J(x-y)$ such that $\<w, x^*> \leq 0$.
  We deduce from \Cref{lem:separation} that $x - y = 0$, and consequently that $\FP(g)$ is
  a singleton. 
\end{proof}

\subsection{Generic uniqueness}

Semicontinuous mappings are ``generically'' continuous.
Before making this fact precise, let us explain the terminology.
A subset $\mathcal{B}$ of a set $\Wcal \subset \Xcal$ is nowhere dense in $\Wcal$ if
the interior relative to $\Wcal$ of the closure relative to $\Wcal$ of $\mathcal{B}$ is empty:
\[
  \intr_\Wcal \left( \clo_\Wcal (\mathcal{B}) \right) = \emptyset .
\]
A subset of $\Wcal \subset \Xcal$ is {\em meager} in $\Wcal$ if it is a countable union of
nowhere dense subsets in $\Wcal$.
Dually, the complement of a meager set in $\Wcal$ is the intersection of countably many subsets
with dense interiors relative to $\Wcal$.
Note that when $\Wcal$ is open or closed in $\Xcal$, a fortiori when $\Wcal = \Xcal$,
the complement of any meager set in $\Wcal$ is dense in $\Wcal$.

\begin{theorem}[{Generic continuity of set-valued mappings}, see {\cite[Thm.~5.55]{RW09}}]
  \label{thm:generic-continuity}
  Let $A: \Xcal \toto \Xcal$ be a closed-valued mapping.
  If $A$ is outer (resp., inner) semicontinuous relative to $\Wcal \subset \Xcal$
  (see \Cref{def:semicontinuity}), then the set of points where $A$ fails to be continuous
  relative to $\Wcal$ is meager in $\Wcal$.
\end{theorem}

We know from \Cref{lem:FP-osc} that the fixed-point mapping is o.s.c.\ and closed-valued.
Hence, a straightforward combination of \Cref{thm:FP-uniqueness} and
\Cref{thm:generic-continuity} leads to the following.

\begin{theorem}[Generic uniqueness of the fixed point]
  \label{thm:FP-generic-uniqueness}
  Let $T$ be a nonexpansive self-map on a finite-dimensional real vector
  space $(\Xcal,\norm{})$.
  Let $\FP: \Xcal \toto \Xcal$ be the fixed-point mapping defined in \labelcref{eq:FP-mapping}.
  Then the set of points $g \in \intr \left( \dom(\FP) \right)$ where $g+T$ fails to have
  a unique fixed point is meager in $\intr \left( \dom(\FP) \right)$.
  In particular, when $\dom(\FP) = \Xcal$, the set of points $g \in \Xcal$ where
  $g+T$ has a unique fixed point is dense in $\Xcal$.
  \qed
\end{theorem}

\section{Application to Shapley operators and stochastic games}
\label{sec:applications}

In this section, we first apply the results of the previous one to the case of monotone
additively homogeneous operators.
Since this includes in particular the case of Shapley operators, a proof of
\Cref{thm:ergodicity-slice-spaces} readily follows.
Then, we use a combinatorial criterion for the boundedness in Hilbert's seminorm of
all the slice spaces of any monotone additively homogeneous operator (announced in
\cite{AGH15-CDC}) to derive the ergodicity condition in terms of dominions which appears
in \Cref{thm:ergodicity-dominions}.

\subsection{From fixed-point to ergodicity problems}
\label{sec:fixed-point-equivalence}

In this subsection, we show how the solvability of the ergodic equation \labelcref{eq:acoe}
for a monotone additively homogeneous map is equivalent to a fixed point problem involving
a nonexpansive map in some finite-dimensional normed space.
We then adapt the results of \Cref{sec:FP-problems} to the former problem
(solvability of the ergodic equation).
This provides in particular a characterization of ergodic stochastic games,
stated in \Cref{thm:ergodicity-slice-spaces} which .

\medskip
Let $\TP^n := \R^n / \R \unit$ be the quotient space of $\R^n$ by the subspace $\R \unit$,
i.e.,  the set of equivalence classes over $\R^n$ by the following
relation: $x \sim y$ if there exists $\alpha \in \R$ such that $x-y = \alpha \unit$. 
We denote by $[x]$ the equivalence class of any vector $x \in \R^n$ modulo the relation $\sim$.
Observe that for any $x, y \in \R^n$ such that $x \sim y$, we have $\Hilbert{x} = \Hilbert{y}$,
where Hilbert's seminorm $\Hilbert{}$ is defined in \labelcref{eq:Hilbert-seminorm}.
Thus, Hilbert's seminorm can be quotiented into a map, which we denote by $\Hnorm$, over $\TP^n$.
Furthermore, this quotiented seminorm is now a norm since $\Hnorm([x]) = \Hilbert{x}= 0$
if and only if $x \sim 0$, that is, $[x] = 0$.
This makes $(\TP^n, \Hnorm)$ a finite-dimensional normed vector space.

Any monotone additively homogeneous map $T: \R^n \to \R^n$ can be quotiented into a map
$[T]: \TP^n \to \TP^n$, sending any equivalence class $[x]$ to $[T(x)]$.
Furthermore, it is known that such a map $T$ is nonexpansive with respect to
Hilbert's seminorm (see e.g., \cite{GG04}), that is,
\[
  \Hilbert{T(x)-T(y)} \leq \Hilbert{x-y} , \quad \forall x, y \in \R^n .
\]
Hence the quotiented map $[T]$ is nonexpansive with respect to the norm $\Hnorm$.
Observe further that $(\lambda,u) \in \R \times \R^n$ is a solution of \Cref{eq:acoe} if
and only if the equivalence class $[u] \in \TP^n$ is a fixed point of $[T]$.
Moreover, the uniqueness of the fixed point of $[T]$ is equivalent to the uniqueness
up to an additive constant of the solution of \labelcref{eq:acoe}, meaning that any solution of
\labelcref{eq:acoe} is of the form $(\lambda, u + \alpha \unit)$ with $\alpha \in \R$.

\medskip
The latter considerations allow us to adapt the results of \Cref{sec:FP-problems} to
the setting of monotone additively homogeneous maps, which includes in particular
the case of Shapley operators of stochastic games.
The next theorem, from which readily follows \Cref{thm:ergodicity-slice-spaces},
gives stability conditions for the solvability of the ergodic equation~\labelcref{eq:acoe}.

\begin{theorem}[Stability of the ergodic equation]
  \label{thm:stability-acoe}
  Let $T: \R^n \to \R^n$ be a monotone additively homogeneous map.
  The following assertions are equivalent.
  \begin{enumerate}
    \item \Cref{eq:acoe} has a solution for all maps $g+T$ with $g \in \R^n$;
      \label{it:stability-acoe-i}
    \item \Cref{eq:acoe} has a solution for all monotone additively homogeneous maps
      $G: \R^n \to \R^n$ such that $\sup_{x \in \R^n} \Hilbert{G(x)-T(x)} < \infty$;
      \label{it:stability-acoe-ii}
    \item the set $\Dcal_\alpha^\textup{H}(T) :=
      \{ x \in \R^n \mid \Hilbert{x-T(x)} \leq \alpha \}$ is bounded in Hilbert's seminorm
      for all $\alpha \geq 0$;
      \label{it:stability-acoe-iii}
    \item the slice space $\Scal_\alpha^\beta(T) = \{ x \in \R^n \mid
      \alpha \unit + x \leq T(x) \leq \beta \unit + x \}$
      is bounded in Hilbert's seminorm for all $\alpha,\beta \in \R$.
      \label{it:stability-acoe-iv}
  \end{enumerate}
\end{theorem}

\begin{proof}
  The equivalence between \Cref{it:stability-acoe-i,it:stability-acoe-ii,it:stability-acoe-iii}
  is a straightforward application of \Cref{cor:stability-FP} to the nonexpansive map $[T]$
  on $(\TP^n,\Hnorm{})$.

  Second, it follows from the definition of Hilbert's seminorm that
  $\Scal_\alpha^\beta(T) \subset \Dcal_\gamma^\textup{H}(T)$ whenever $\beta-\alpha \leq \gamma$.
  Hence \labelcref{it:stability-acoe-iii}~implies~\labelcref{it:stability-acoe-iv}.

  Finally, we already know that
  \labelcref{it:stability-acoe-iv}~implies~\labelcref{it:stability-acoe-i}.
  It has been proved in \Cref{sec:ergodic-stochastic-games} (this is an easy consequence
  of \Cref{thm:solvability-acoe}).
\end{proof}

We now address the uniqueness up to an additive constant of the solution to \Cref{eq:acoe}.
From \Cref{thm:FP-generic-uniqueness} we get the following corollary.

\begin{corollary}[Generic uniqueness of the solution of the ergodic equation]
  \label{thm:uniqueness-acoe}
  Let $T: \R^n \to \R^n$ be a monotone additively homogeneous map.
  Assume that the ergodic equation~\labelcref{eq:acoe} is solvable for all maps $g+T$
  with $g \in \R^n$.
  Then, the set of vectors $g \in \R^n$ for which the ergodic equation \labelcref{eq:acoe}
  with $g+T$ fails to have a unique solution up to an additive constant is meager.
\end{corollary}

\begin{proof}
  Since $[g]+[T] = [g+T]$ has a fixed point for all $[g] \in \TP^n$, we know from
  \Cref{thm:FP-generic-uniqueness} that the set of points $[g]$ where $[g]+[T]$ fails to have
  a unique fixed point is meager in $\TP^n$.
  Let $\bigcup_{k \in \N} B_k$ be this set, where for all $k \in \N$, $B_k$ is nowhere
  dense in $\TP^n$.
  Denote by $\pi: \R^n \to \TP^n$ the quotient map.
  Then the set of points $g \in \R^n$ where \Cref{eq:acoe} fails to have a unique solution
  up to an additive constant for $g+T$ is $\pi^{-1} \left( \bigcup_{k \in \N} B_k \right) =
  \bigcup_{k \in \N} \pi^{-1} (B_k)$.

  To complete the proof, we need to show that all the sets $\pi^{-1} (B_k)$, $k \in \N$,
  are nowhere dense in $\R^n$.
  To that purpose, we next prove that the preimage by $\pi$ of any set $B$ nowhere
  dense in $\TP^n$ is nowhere dense in $\R^n$.
  Let $\Ucal$ be an open set included in $\clo \left( \pi^{-1}(B) \right)$, the closure
  of $\pi^{-1} (B)$.
  Since $\pi$ is continuous, we have
  $\clo \left( \pi^{-1}(B) \right) \subset \pi^{-1} \left( \clo(B) \right)$, which yields
  \[
    \pi(\Ucal) \subset \pi \left( \pi^{-1} \left( \clo(B) \right) \right) = \clo(B).
  \]
  Furthermore, since $\pi$ is a surjective continuous linear operator, then it is an open
  map according to the Banach-Schauder theorem, and so $\pi(\Ucal)$ is open.
  We deduce that $\pi(\Ucal)$ is empty, since $B$ is nowhere dense in $\TP^n$.
  As a consequence, we finally have $\Ucal = \emptyset$, which proves that
  $\intr \left( \clo \left( \pi^{-1}(B) \right) \right) = \emptyset$, i.e.,
  $\pi^{-1} (B)$ is nowhere dense in $\R^n$.
\end{proof}

\subsection{Geometric ergodicity conditions}

In \cite{AGH15-CDC}, Akian {\em et al.}\ gave a combinatorial criterion (in terms of
hypergraph) for the boundedness in Hilbert's seminorm of all the slice spaces of
a monotone additively homogeneous map $T: \R^n \to \R^n$ (see also the preprint \cite{AGH18}).
This criterion boils down to the following result, where, for any set $L \subset \state$,
we denote by $\unit_L$ the vector in $\R^n$ with entries equal to $1$ on $L$ and $0$ elsewhere.

\begin{proposition}[Boundedness of all slice spaces, \cite{AGH15-CDC}]
  \label{prop:boundedness-slice-spaces}
  All the slice spaces of a monotone additively homogeneous map $T: \R^n \to \R^n$ are bounded
  in Hilbert's seminorm if and only if there exist two disjoint nonempty subsets of $[n]$,
  $I$ and $J$, satisfying, respectively,
  \[
    \forall i \in I, \enspace
    \lim_{\kappa \to -\infty} T_i( \kappa \, \unit_{\state \setminus I} ) > -\infty
    \quad \text{and} \quad
    \forall j \in J, \enspace
    \lim_{\kappa \to +\infty} T_j( \kappa \, \unit_{\state \setminus J} ) < +\infty .
  \]
\end{proposition}

We next show that the latter conditions can be interpreted in game-theoretical terms, involving
dominions (\Cref{def:dominions}), when $T$ arises as the Shapley operator of a stochastic game
satisfying \Cref{asm:compact-continuous}.
To that purpose, let us first give a simpler characterization of dominions in terms of
pure actions.

To prove the following lemma, we shall use the notion of {\em support} of a Borel measure $\mu$
on a Borel space $X$.
We recall that if this support exists, it is the unique nonempty closed set, denoted by
$\supp \mu$, satisfying
\begin{itemize}
  \item $\mu(X \setminus \supp \mu) = 0$;
  \item if $U \subset X$ is open and $U \cap \supp \mu \neq \emptyset$, then
    $\mu(U \cap \supp \mu) > 0$;
\end{itemize}
(see e.g., \cite{AB06}).
Note that if $\mu$ is a Borel probability measure, then it is regular (see Theorem~12.7, {\em ibid.})
which implies that $\supp \mu$ exists (see Theorem~12.14, {\em ibid.}).

\begin{lemma}[Characterization of dominions]
  \label{lem:dominion-characterization}
  Let $\Gamma$ be a stochastic game satisfying \Cref{asm:compact-continuous}.
  A subset of states $D$ is a dominion of player \Max (resp., \Min) if and only if
  \begin{align}
    \label{eq:dominion-max}
    \forall i \in D , & \quad \exists \bar{a} \in A_i , \quad
    \forall b \in B_i , \quad p(D \mid i,\bar{a},b) = 1 \\
    \label{eq:dominion-min}
    \big( \text{resp.,} \quad
    \forall i \in D , & \quad \exists \bar{b} \in B_i , \quad
    \forall a \in A_i , \quad p(D \mid i,a,\bar{b}) = 1 \big) .
  \end{align}
\end{lemma}

\begin{proof}
  First assume that \labelcref{eq:dominion-max} holds and consider any pure stationary strategy
  $\bar{\sigma}$ of player \Max that assigns to every current position $i \in D$ an action
  $\bar{a}_i \in A_i$ satisfying \labelcref{eq:dominion-max}.
  Let $\tau$ be any strategy of player \Min and let $\{i_k\}_{k \geq 0}$ be a sequence of
  states generated by the pair $(\bar{\sigma}, \tau)$ with an initial state $i_0 \in D$.
  Conditioning on the first exit time, the probability that the state leaves $D$ at some
  stage is
  \[
    \proba \left( \{i_k\} \not \subset D \right) = \sum_{k \geq 0}
    \proba \left( i_{k+1} \notin D \mid \{i_\ell\}_{0 \leq \ell \leq k} \subset D \right) \,
    \proba \left( \{i_\ell\}_{0 \leq \ell \leq k} \subset D \right) .
  \]
  With the particular choice of $\bar{\sigma}$ that we made, we have, for every stage $k \geq 0$,
  \[
    \proba \left( i_{k+1} \notin D \mid \{i_\ell\}_{0 \leq \ell \leq k} \subset D \right) \leq
    \max_{b \in B_{i_k}} p(\state \setminus D \mid i_k, \bar{a}_{i_k}, b) = 0 .
  \]
  Hence, the probability that the state leaves $D$ at some stage is 0, that is, $D$ is
  a dominion of player \Max.

  Conversely, assume that $D$ is a dominion of player \Max.
  Then, by definition of dominions, for every initial state $i \in D$, there exists
  a mixed action $\bar{\mu} \in \Delta(A_i)$ of player \Max such that for all mixed actions
  of player \Min, hence for all pure actions $b \in B_i$, the probability
  $p( D \mid i, \bar{\mu}, b)$ that the next state remains in $D$ is equal to 1.

  Take $\bar{a} \in \supp \bar{\mu}$ and assume, working toward a contradiction, that
  there exists some $\bar{b} \in B_i$ such that $p( D \mid i, \bar{a}, \bar{b}) < 1$.
  Since the function $a \mapsto p(\state \setminus D \mid i, a, \bar{b})$ is continuous
  (see \Cref{rem:continuity}) and positive in $\bar{a}$, then there exists a scalar
  $\varepsilon > 0$ and an open neighborhood $\Ucal$ of $\bar{a}$ such that
  $p( \state \setminus D \mid i, a, \bar{b}) \geq \varepsilon$
  for all $a \in \Ucal$.
  This implies that
  \[
    p( \state \setminus D \mid i, \bar{\mu}, \bar{b}) \geq
    \int_{\Ucal} p( \state \setminus D \mid i, a, \bar{b}) \, d\bar{\mu}(a) \geq
    \varepsilon \, \bar{\mu}(\Ucal) .
  \]
  However, since $\Ucal \cap \supp \bar{\mu}$ is nonempty (it contains $\bar{a}$), it follows from
  the definition of the support of a measure that
  $\bar{\mu}(\Ucal) \geq \bar{\mu}(\Ucal \cap \supp \bar{\mu}) > 0$.
  This implies that $p(\state \setminus D \mid i, \bar{\mu}, \bar{b}) > 0$, or equivalently
  $p(D \mid i, \bar{\mu}, \bar{b}) < 1$, a contradiction.
  So, for all $b \in B_i$, we have $p(D \mid i, \bar{a}, b) = 1$,
  hence \labelcref{eq:dominion-max} is satisfied.
  With dual arguments we can show that the dominions of player \Min are characterized by
  \labelcref{eq:dominion-min}.
\end{proof}

\begin{example}
  Consider the game introduced in \Cref{ex:dominions}, whose transition function is defined,
  for all actions $a \in [0,1]$ of player \Max and $b \in [0,1]$ of player \Min, by
  \[
    p(1,a,b) =
    \begin{pmatrix}
      \gamma a \\ b \, (1 - \gamma a) \\ (1-b) \, (1 - \gamma a)
    \end{pmatrix} \mkern-5mu , \;
    p(2,a,b) =
    \begin{pmatrix}
      \gamma b \\ a \, (1 - \gamma b) \\ (1-a) \, (1 - \gamma b)
    \end{pmatrix} \mkern-5mu , \;
    p(3,a,b) =
    \begin{pmatrix}
      0 \\ 0 \\ 1
    \end{pmatrix} \mkern-5mu ,
  \]
  where $\gamma$ is a fixed parameter in $(0,1)$.
  With \Cref{lem:dominion-characterization} in mind, it is now straightforward to check that
  the dominions of player \Max are $\{3\}$ and $\{1,2,3\}$, whereas the dominions of
  player \Min are $\{3\}$, $\{1,3\}$, $\{2,3\}$ and $\{1,2,3\}$.
\end{example}

\begin{remark}
  \label{rem:pure-mixed-actions}
  We easily deduce from the proof of the latter result that a dominion $D$ of player \Max
  (resp., \Min) in $\Gamma$ is also characterized by the following equality:
  \begin{align*}
    \max_{\mu \in \Delta(A_i)} \min_{\nu \in \Delta(B_i)} & p( D \mid i, \mu, \nu) = 1 \\
    \big( \text{resp.,} \quad
    \max_{\mu \in \Delta(A_i)} \min_{\nu \in \Delta(B_i)} &
    p( \state \setminus D \mid i, \mu, \nu) = 0 \big) ,
  \end{align*}
  where the minimum and the maximum commutes.
\end{remark}

We can now identify the sets of states satisfying one of the asymptotic properties
in \Cref{prop:boundedness-slice-spaces}.
This characterization (which is established in the following lemma) combined with
\Cref{prop:boundedness-slice-spaces} and \Cref{thm:ergodicity-slice-spaces}, yields
the dominion condition for ergodicity of stochastic games announced
in \Cref{thm:ergodicity-dominions}.

\begin{lemma}[Dominions and asymptotic behavior of the Shapley operator]
  \label{lem:dominions}
  Let $\Gamma$ be a stochastic game satisfying \Cref{asm:compact-continuous},
  with Shapley operator $T$.
  A set $D \subset \state$ is a dominion of player \Max (resp., \Min) in $\Gamma$ if
  and only if
  \begin{align}
    \label{eq:abstract-dominion-max}
    \forall i \in D, & \quad
    \lim_{\kappa \to -\infty} T_i( \kappa \, \unit_{\state \setminus D} ) > -\infty , \\
    \label{eq:abstract-dominion-min}
    \big( \text{resp.,} \quad
    \forall i \in D, & \quad
    \lim_{\kappa \to +\infty} T_i( \kappa \, \unit_{\state \setminus D} ) < +\infty \big) .
  \end{align}
\end{lemma}

\begin{proof}
  Assume first that $D$ is a dominion of player \Max in $\Gamma$.
  Then, according to \Cref{lem:dominion-characterization}, for every $i \in D$
  there exists an action $\bar{a} \in A_i$ such that $p(D \mid i, \bar{a}, b) = 1$
  for all $b \in B_i$.
  Hence, for all real numbers $\kappa \leq 0$ we have
  \begin{equation*}
    \begin{split}
      T_i(\kappa \unit_{\state \setminus D}) \geq & \min_{\nu \in \Delta(B_i)}
      \bigg\{ r(i,\bar{a},\nu) + \sum_{\ell \notin D} \kappa \, p(\ell \mid i,\bar{a},\nu) \bigg\} \\
      = & \min_{\nu \in \Delta(B_i)} \Big\{ r(i,\bar{a},\nu) +
      \kappa \, p(\state \setminus D \mid i,\bar{a},\nu) \Big\} \\
      = & \min_{\nu \in \Delta(B_i)} r(i,\bar{a},\nu) .
    \end{split}
  \end{equation*}
  Since the map $\nu \mapsto r(i,\bar{a},\nu)$ is l.s.c.\ on $\Delta(B_i)$ and since
  the latter set is compact, we deduce that the right-hand side of the above inequality
  is finite.
  This shows that \labelcref{eq:abstract-dominion-max} is true.

  Conversely, assume that \labelcref{eq:abstract-dominion-max} holds true.
  Let $i \in D$ and let $\nu \in \Delta(B_i)$ be a mixed action of player \Min.
  Since the map $\mu \mapsto r(i,\mu,\nu)$ is u.s.c.\ on $\Delta(A_i)$ and since
  the latter set is compact, there exists some $r^* \in \R$ such that
  $r(i,\mu,\nu) \leq r^*$ for all $\mu \in \Delta(A_i)$.
  Then we have for all $\kappa \leq 0$
  \begin{equation*}
    \begin{split}
      T_i(\kappa \, \unit_{\state \setminus D}) & \leq \max_{\mu \in \Delta(A_i)} \bigg\{
        r(i,\mu,\nu) + \sum_{\ell \notin D} \kappa \, p(\ell \mid i,\mu,\nu) \bigg\} \\
        & \leq r^* + \max_{\mu \in \Delta(A_i)} \bigg\{
          \sum_{\ell \notin D} \kappa \, p(\ell \mid i,\mu,\nu) \bigg\} \\
          & = r^* + \kappa \, \min_{\mu \in \Delta(A_i)}
          p(\state \setminus D \mid i,\mu,\nu) .
        \end{split}
      \end{equation*}
  The left-hand term in the latter inequality being lower-bounded, we deduce that,
  for all $\nu \in \Delta(B_i)$,
  \[
    \min_{\mu \in \Delta(A_i)} p(\state \setminus D \mid i,\mu,\nu) \leq 0
  \]
  or, equivalently,
  \[
    \max_{\mu \in \Delta(A_i)} p(D \mid i,\mu,\nu) \geq 1 .
  \]
  Since the above maximum is also lower than $1$, we have
  \[ 
    \min_{\nu \in \Delta(B_i)} \max_{\mu \in \Delta(A_i)} p(D \mid i,\mu,\nu) =
    \max_{\mu \in \Delta(A_i)} \min_{\nu \in \Delta(B_i)} p(D \mid i,\mu,\nu) = 0 .
  \]
  Thus, \Cref{rem:pure-mixed-actions} yields that $D$ is a dominion of player \Max in $\Gamma$.
  With dual arguments we can show that the same is true for the dominions of player \Min.
\end{proof}

\Cref{thm:ergodicity-dominions} shows that the ergodicity of a stochastic game $\Gamma$ is
a structural property, in the sense that it only depends on the transition probabilities.
The following immediate corollary makes this observation even clearer.
However, we draw the attention of the reader to the fact that \cref{asm:compact-continuous} is
not only sufficient but also necessary for this characteristic to hold, as already explained in
\cref{rem:structural-property} and illustrated with \cref{ex:ergodicity-perfect-info}.

Before stating the result, let us fix some notation.
First, we shall say for brevity that a function $g: K \to \R$ is {\em u.s.c./l.s.c.}\
if, for all $(i,a,b) \in K$, $g(i,\cdot,b)$ is u.s.c.\ and $g(i,a,\cdot)$ is l.s.c.
Second, we next consider a parametric family of stochastic games for which all the data
are fixed except the payoff function $r$.
In order to highlight this dependency, we shall use the notation
$\Gamma(r)$ to refer to the game $(\state,A,B,K_A,K_B,r,p)$.

\begin{corollary}
  Let $\Gamma(r)$ be a stochastic game satisfying \Cref{asm:compact-continuous}.
  The following assertions are equivalent:
  \begin{enumerate}
    \item $\Gamma(r)$ is ergodic;
    \item $\Gamma(0)$, i.e., the game with all payoffs equal to $0$, is ergodic;
    \item for all Borel measurable u.s.c./l.s.c.\ payoff functions $\tilde{r}: K \to \R$
      which are bounded below or above, the ergodic equation~\labelcref{eq:acoe} is solvable for
      the Shapley operator of $\Gamma(\tilde{r})$.
  \end{enumerate}
\end{corollary}

\bibliographystyle{amsalpha}
\bibliography{references}

\end{document}